\documentclass[11pt]{amsart}
\usepackage{amsmath,amssymb,amsthm, a4wide,scalerel}
\usepackage{mathabx}
\usepackage{color}
\usepackage{xcolor,graphicx}
\usepackage{mathrsfs}
\usepackage{tikz,pgfplots}
\usepackage{subcaption}
\usepackage[normalem]{ulem}
\usepackage{dsfont}
\usepackage{comment}
\usepackage{enumitem}
\usepackage{ctable} % for \specialrule command
\usepackage{hyperref}
\usepackage{cleveref}

\newtheorem{theorem}{Theorem}
\newtheorem{lemma}[theorem]{Lemma}
\newtheorem{proposition}[theorem]{Proposition}
\newtheorem{remark}[theorem]{Remark}
\newtheorem{corollary}[theorem]{Corollary}

\newtheorem{assumption}{Assumption}

\crefname{theorem}{Theorem}{Theorems}
\crefname{lemma}{Lemma}{Lemmas}
\crefname{proposition}{Proposition}{Propositions}
\crefname{remark}{Remark}{Remarks}
\crefname{corollary}{Corollary}{Corollaries}
\crefname{defi}{Definition}{Definitions}
\crefname{hypo}{Hypothesis}{Hypotheses}
\crefname{assumption}{Assumption}{Assumptions}
\crefname{example}{Example}{Examples}

\newcommand{\dt}{\tau}

\newcommand{\R}{{\mathbb R}}
\newcommand{\N}{{\mathbb N}}
\newcommand{\E}{{\mathbb E}}

\providecommand{\norm}[1]{\left\lVert#1\right\rVert}
\newcommand*\dd{\mathop{}\!\mathrm{d}}

\newcommand{\sI}{{\upshape PP(0.5)}}
\newcommand{\hosI}{{\upshape PP(1.0)}}

\usepackage{color}

\author[C.-E. Br\'ehier]{Charles-Edouard Br\'ehier}
              \address{Universite de Pau et des Pays de l'Adour, E2S UPPA, CNRS, LMAP, Pau, France}
              \email{charles-edouard.brehier@univ-pau.fr}

\author[D. Cohen]{David Cohen}
              \address{Department of Mathematical Sciences,
              Chalmers University of Technology and University of Gothenburg, 41296~Gothenburg, Sweden}
              \email{\tt david.cohen@chalmers.se}

\begin{document}

\title[First-order positivity preserving scheme for SDE\MakeLowercase{s}]{Analysis of a first-order explicit positivity preserving scheme for a class of scalar SDE\MakeLowercase{s}}

\begin{abstract}
We propose and analyze a first-order numerical scheme 
for a class of scalar It\^o stochastic differential equations with almost surely positive solutions. We construct a new explicit numerical scheme, such that for any choice of the time-step size, the numerical solution is guaranteed to remain almost surely positive. The main result of this article is the first-order strong convergence of the proposed positivity preserving scheme, which is illustrated with numerical experiments and proved rigorously. It is also shown how to adapt the scheme and the results to Stratonovich stochastic differential equations with positive solutions.
\end{abstract}

\maketitle
{\small\noindent

First-order strong convergence. 
}

\section{Introduction}\label{secIntro}

In this article, we consider scalar It\^o stochastic differential equations (SDEs)
\begin{equation}\label{eq:SDEintro}
\left\lbrace
\begin{aligned}
\dd X(t)&=G(X(t))\dd t+g(X(t))\dd B(t),\quad t\ge 0,\\
X(0)&=x_0,
\end{aligned}
\right.
\end{equation}
where the drift and diffusion coefficients satisfy the condition $G(0)=g(0)=0$ (we refer to Section~\ref{secAss} for precise regularity and growth conditions). This ensures that the solution remains nonnegative at all times: if $x_0\ge 0$, then $X(t)\ge 0$ almost surely for all $t\ge 0$. Standard numerical methods such as the Euler--Maruyama and the Milstein schemes do not preserve this property. Positivity preserving schemes have been proposed and studied in previous works, mostly for specific examples of SDEs, with main techniques being, Lamperti or logarithmic transformations, and modifications of classical numerical schemes \cite{MR2341800,MR2367990,MR3082312,MR3248050,MR3331648,MR4268206,MR4475995,MR5051815,MR4847179,MR4888024,MR5004937,MR4993436,MR5043963}. The two main models considered in the literature for construction of positive preserving schemes are the stochastic Susceptible-Infected-Susceptible (SIS) epidemic models, see for instance~\cite{MR4220738,MR4274899,MR4469088,MR4444727,MR4395894,MR4597411,MR4525901,MR4790879}, and the Cox--Ingersoll--Ross (CIR) process in finance, see for instance~\cite{MR2186814,MR2898556,MR3006996,MR3433041,MR3732573,MR4544037}. The two lists of references above are not exhaustive. We emphasize that the SIS and CIR models are not covered by the contribution of this article, due to the assumptions imposed on the coefficients $G$ and $g$. Instead we consider a general class of models and the proposed numerical method can be adapted to a wide range of cases. For simplicity of the exposition, we consider only scalar SDEs but the results can be adapted to some higher-dimensional situations. To conclude this review of the literature, we mention the recent works~\cite{MR4780408,MR4729657,Djurdjevac2026,ana26}, and references therein, for the design and analysis of positivity preserving numerical schemes for stochastic partial differential equations, where both temporal and spatial discretization are considered.

We briefly review the two main ingredients of our construction of positivity preserving schemes, before proposing the new first-order numerical scheme. First, noting that $G(0)=g(0)=0$, the drift and diffusion coefficients $G$ and $g$ can be expressed as $G(x)=xF(x)$ and $g(x)=xf(x)$ for auxiliary mappings $F$ and $f$. Second, for arbitrary parameters $a,\alpha\in\R$ and initial value $y_0\in\R$, the solution to the linear It\^o stochastic differential equation
\[
\left\lbrace
\begin{aligned}
\dd Y(t)&=aY(t)\dd t+\alpha Y(t)\dd B(t),\quad t\ge 0,\\
Y(0)&=y_0,
\end{aligned}
\right.
\]
is given at all times $t\ge 0$ by
\[
Y(t)=\exp\left(at+\alpha B(t)-\frac{\alpha^2 t}{2}\right)y_0.
\]
Combining the two ingredients above between two grid points yields the scheme
\begin{equation}\label{eq:schemedemiintro}
X_{n+1}=\exp\left(F(X_n)\dt+f(X_n)\delta B_n-\frac12 f(X_n)^2\dt\right)X_n,
\end{equation}
with initial value $X_0=x_0$, see Section~\ref{sec:constructionscheme} for a more detailed derivation. Note that similar ideas have been used in \cite{MR4177372,MR4287454,bossy2024,MR4780408,MR4729657} for instance. Observe that the scheme~\eqref{eq:schemedemiintro} is positivity preserving for any choice of the time-step size $\dt$: if $x_0\ge 0$, then $X_n\ge 0$ almost surely for all $n\ge 0$. Numerical experiments show that the scheme~\eqref{eq:schemedemiintro} converges in general with strong order $1/2$ to the solution of~\eqref{eq:SDEintro}.

To achieve higher order of convergence, we propose the following new explicit positivity preserving scheme for~\eqref{eq:SDEintro}: 
\begin{equation}\label{eq:schemeintro}
X_{n+1}=\exp\left(F(X_n)\dt+f(X_n)\delta B_n+\frac12\bigl(f' g\bigr)(X_n)\bigl[\delta B_n^2-\dt]-\frac12 f(X_n)^2\dt\right)X_n,
\end{equation}
with initial value $X_0=x_0$. Like in the standard Milstein scheme applied to~\eqref{eq:SDEintro} compared to the Euler--Maruyama scheme, the appropriate correction term $\frac12\bigl(f' g\bigr)(X_n)\bigl[\delta B_n^2-\dt]$ is introduced in~\eqref{eq:schemeintro} compared to the initial scheme~\eqref{eq:schemedemiintro}. The choice of the correction term to achieve strong order $1$ for the convergence is motivated in Section~\ref{sec:constructionscheme}.

Our main results on the scheme~\eqref{eq:schemeintro} can be summarized as follows. First, for any time-step size $\dt$, we show that the explicit scheme~\eqref{eq:schemeintro} is positivity preserving, see Proposition~\ref{prop:posi}: if $x_0\ge 0$, then $X_n\ge 0$ almost surely for all $n\ge 0$. Second, we show that the strong order of convergence of the scheme~\eqref{eq:schemeintro} to the solution of~\eqref{eq:SDEintro} is equal to $1$, see Theorem~\ref{theoMain}: for arbitrary $p\in[1,\infty)$ and $T\in(0,\infty)$, one has
\begin{equation}\label{eq:errorintro}
\underset{0\le n\le N}\sup~\left(\E\left[|X_n-X(t_n)|^{2p}\right]\right)^{\frac{1}{2p}}\le C_p(T,x_0)\dt.
\end{equation}
We also show that~\eqref{eq:errorintro} provides pathwise strong error estimates and pathwise almost sure convergence with rate of convergence $1-\varepsilon$ with arbitrary $\varepsilon\in(0,1)$, see \Cref{corCV}. The first-order positivity preserving scheme~\eqref{eq:schemeintro} can be adapted to approximate solutions to Stratonovich stochastic differential equations, see Corollary~\ref{corStrato}.

We also provide numerical experiments in Section~\ref{sec:num} to illustrate the main results and the superiority of the new first-order positivity preserving scheme~\eqref{eq:schemeintro} compared to the original positivity preserving scheme~\eqref{eq:schemedemiintro} with strong order $1/2$, and to the classical Euler--Maruyama and Milstein schemes which do not preserve positivity.

A large part of the article is devoted to the proof of the first-order strong error estimates~\eqref{eq:errorintro} above, with rigorous and detailed treatments of auxiliary error terms. The error is decomposed into relevant auxiliary error terms in Section~\ref{secDec}, and the main cancellation which allows to achieve the first-order convergence is highlighted there. Note that the proof of~\eqref{eq:errorintro} requires to preliminarily establish strong error estimates with order $1/2$ for all $p\in[1,\infty)$.

We also emphasize that the proposed first-order positivity preserving scheme~\eqref{eq:schemeintro} can be compared to the first-order domain preserving scheme constructed and analyzed in~\cite{DP}. In particular, from \cite{DP}, one can deduce how to generalize the definition of the first-order positivity preserving scheme for systems of It\^o and Stratonovich stochastic differential equations in arbitrary dimension $d$, with solutions taking values in $[0,\infty)^d$. 

This article is organized as follows. Section~\ref{secAss} presents the theoretical framework and the considered It\^o and Stratonovich stochastic differential equations. The proposed class of positivity preserving numerical schemes is presented in Section~\ref{secPosi}: its construction is provided in Section~\ref{sec:constructionscheme}, the main results are stated in Section~\ref{sec:main} and illustrated in Section~\ref{sec:num}. Section~\ref{secAux} presents auxiliary results, and finally Section~\ref{secProof} presents the detailed proof of Theorem~\ref{theoMain}.

\section{Setting}\label{secAss}
In this section, we present the setting of the paper and provide some useful results on the analytical solution of the problem.

\subsection{Notation}

For a continuous and bounded mapping $h\colon\R^+\to\R$, set $\|h\|_\infty=\underset{x\in\R^+}\sup~|h(x)|$.

Let $(\Omega,\mathcal{F},\mathbb{P})$ denote a probability space equipped with a filtration $\bigl(\mathcal{F}_t\bigr)_{t\ge 0}$, where we consider a standard real-valued Wiener process $\bigl(B(t)\bigr)_{t\ge 0}$. Let us recall the Burkholder--Davis--Gundy (BDG) inequality:
for all $T\in(0,\infty)$ and all $p\in[1,\infty)$,
there exists $C_p\in(0,\infty)$ and $C_p(T)=C_pT^{p-1}\in(0,\infty)$ such that for any predictable scalar stochastic process $\bigl(\phi(t)\bigr)_{t\in[0,T]}$ and for all $t\in[0,T]$, one has
\begin{equation}\label{eq:BDG}
\E\left[\left|\int_{0}^{t}\phi(s)\dd B(s)\right|^{2p}\right]\le C_p\E\left[\left(\int_{0}^{t}|\phi(s)|^2 \dd s\right)^{p}\right]\le
C_p(T)\int_{0}^{T}\E\left[|\phi(s)|^{2p}\right]\dd s.
\end{equation}
In addition, one has the following discrete-time version of the BDG inequality which is employed in the analysis of the proposed numerical scheme. For all $T\in(0,\infty)$ and all $p\in[1,\infty)$, there exists $C_p(T)\in(0,\infty)$ such that for any $N\in\N$, any family $\bigl(t_n\bigr)_{0\le n\le N}$ of points such that $0=t_0<t_1<\ldots<t_{N-1}<t_N=T$, and any family $\bigl(\phi_n\bigr)_{0\le n\le N}$ of predictable processes, for all $n\in\{1,\ldots,N\}$ one has
\begin{equation}\label{eq:BDGdiscrete}
\E\left[\left|\sum_{k=0}^{n-1}\int_{t_k}^{t_{k+1}}\phi_k(s)\dd B(s)\right|^{2p}\right]\le C_p(T)\sum_{k=0}^{n-1}\int_{t_k}^{t_{k+1}}\E\left[|\phi_k(s)|^{2p}\right]\dd s.
\end{equation}

\subsection{It\^o stochastic differential equation}

In this article, we consider the scalar It\^o stochastic differential equation
\begin{equation}\label{eq:sde}
\left\lbrace
\begin{aligned}
\dd X(t)&=G(X(t))\dd t+g(X(t))\dd B(t),\quad t\ge 0,\\
X(0)&=x_0,
\end{aligned}
\right.
\end{equation}
where the initial value $x_0$, the drift coefficient $G$ and the diffusion coefficient $g$ satisfy the assumptions below.

\begin{assumption}\label{assInit}
The initial value $x_0$ is deterministic and satisfies $x_0\ge 0$.
\end{assumption}

\begin{assumption}\label{assCoeff}
The mappings $G,g\colon\R^+\to\R$ are of class $\mathcal{C}^2$ with bounded first and second order derivatives.
Furthermore, one has
\begin{equation}\label{eq:Gg0}
G(0)=g(0)=0.
\end{equation}
\end{assumption}
Note that the mappings $G$ and $g$ then satisfy the following linear growth conditions:
\begin{equation}\label{eq:Gg-growth}
|G(x)|\le \|G'\|_\infty x,\quad |g(x)|\le \|g'\|_\infty x,\qquad \forall~x\in\R^+.
\end{equation}
Furthermore, one has $(g'g)'=g''g+(g')^2$. The first and second order derivatives $g'$ and $g''$ of the mapping $g$ are bounded on $\R^+$ owing to \Cref{assCoeff}, however the mapping $g$ is not assumed to be bounded but satisfies the inequality~\eqref{eq:Gg-growth}. Therefore there exists $C\in(0,\infty)$ such that
\begin{equation}\label{eq:boundg'g}
\big|(g'g)(x_2)-(g'g)(x_1)\big|\le C\bigl(1+x_1+x_2\bigr),\qquad \forall~x_1,x_2\in[0,\infty).
\end{equation}

The next result shows that the stochastic differential equation~\eqref{eq:sde} admits a unique solution, which takes valuesin $[0,\infty)$,
and also provides moment bounds for this solution.

\begin{proposition}\label[proposition]{propMomEx}
Consider the stochastic differential equation~\eqref{eq:sde} under Assumptions~\ref{assInit}~and~\ref{assCoeff}.

It admits a unique solution $\bigl(X(t)\bigr)_{t\ge 0}$ which takes values in $[0,\infty)$: almost surely, $X(t)\ge 0$ for all $t\ge 0$.

Furthermore, for all $T\in(0,\infty)$ and $p\in[1,\infty)$, there exists $C_p(T,x_0)\in(0,\infty)$ such that
\begin{equation}\label{momentExact}
\underset{t\in[0,T]}\sup~\E\left[|X(t)|^{2p}\right]\le C_p(T,x_0).
\end{equation}
\end{proposition}

\begin{proof}
Owing to Assumption~\ref{assInit}, the mappings $G$ and $g$ are Lipschitz continuous on $[0,\infty)$, and they satisfy the condition~\eqref{eq:Gg0}. Combining a standard fixed point argument and a
comparison principle (see~\cite{yamada}, for instance) shows that there exists a unique solution $\bigl(X(t)\bigr)_{t\ge 0}$ and that almost surely $X(t)\ge 0$ for all $t\ge 0$.

It remains to prove the moment bounds~\eqref{momentExact}. Let $T\in(0,\infty)$ and $p\in[1,\infty)$. For all $t\ge 0$, one has
\[
X(t)=x_0+\int_{0}^{t}G(X(s))\dd s+\int_{0}^{t}g(X(s))\dd B(s).
\]
Applying the H\"older inequality, the BDG inequality~\eqref{eq:BDG}, and the linear growth property~\eqref{eq:Gg-growth} of the mappings $G$ and $g$, for all $t\in[0,T]$ one obtains
\begin{align*}
\E\left[|X(t)|^{2p}\right]&\le C_p(T)\Bigl(|x_0|^{2p}+\int_{0}^{t}\E\left[|G(X(s))|^{2p}\right]\dd s+\int_{0}^{t}\E\left[|g(X(s))|^{2p}\right]\dd s\Bigr)\\
&\le C_p(T)\Bigl(|x_0|^{2p}+\|G'\|_\infty^{2p}\int_{0}^{t}\E\left[|X(s)|^{2p}\right]\dd s+\|g'\|_\infty^{2p}\int_{0}^{t}\E\left[|X(s)|^{2p}\right]\dd s\Bigr).
\end{align*}
Applying the Gr\"onwall lemma then provides the moment bounds \eqref{momentExact}. This concludes the proof of ~\Cref{propMomEx}.
\end{proof}

The definition of positivity preserving schemes is based on an equivalent formulation of the SDE~\eqref{eq:sde} presented below.

Define the mappings $F,f\colon\R^+\to\R$ by
\[
F(x)=\int_{0}^{1}G'(\theta x)\dd \theta,\qquad f(x)=\int_{0}^{1}g'(\theta x)\dd \theta,\qquad \forall~x\in\R^+.
\]
Due to the condition~\eqref{eq:Gg0}, one has the factorization property
\begin{equation}\label{eq:facto}
G(x)=xF(x),\quad g(x)=xf(x),\qquad \forall~x\in\R^+.
\end{equation}
Moreover, owing to Assumption~\ref{assCoeff} the mappings $F$ and $f$ are of class $\mathcal{C}^1$, and satisfy
\begin{align}
\|F\|_\infty \le \|G'\|_\infty<\infty&,\quad \|F'\|_\infty \le \|G''\|_\infty<\infty,\label{eq:boundF}\\
\|f\|_\infty \le \|g'\|_\infty<\infty&,\quad \|f'\|_\infty \le \|g''\|_\infty<\infty.\label{eq:boundf}
\end{align}
Moreover, owing to the identity $g'f=f^2+f'g$ following from~\eqref{eq:facto}, the mapping $f'g=g'f-f^2$ is also bounded:
\begin{equation}
\|f'g\|_\infty\le 2\|g'\|_\infty^2<\infty.\label{eq:boundf'g}
\end{equation}

Based on the factorization~\eqref{eq:facto} of the coefficients $G$ and $g$, we thus consider the following equivalent formulation of the SDE~\eqref{eq:sde}
\begin{equation}\label{eq:sdef}
\left\lbrace
\begin{aligned}
\dd X(t)&=F(X(t))X(t)\dd t+f(X(t))X(t)\dd B(t),\quad t\ge 0,\\
X(0)&=x_0.
\end{aligned}
\right.
\end{equation}

\subsection{Stratonovich stochastic differential equation}

In this article, we mainly focus on stochastic differential equations with It\^o interpretation of the noise, however considering the Stratonovich interpretation is helpful to explain how the first-order positivity preserving scheme is constructed below.

Given the drift and diffusion coefficients $G$ and $g$, consider the Stratonovich stochastic differential equation
\begin{equation}\label{eq:sdeStrato}
\left\lbrace
\begin{aligned}
\dd \mathbf{X}(t)&=G(\mathbf{X}(t))\dd t+g(\mathbf{X}(t))\circ \dd B(t),\quad t\ge 0,\\
\mathbf{X}(0)&=x_0.
\end{aligned}
\right.
\end{equation}
The Stratonovich SDE~\eqref{eq:sdeStrato} admits the equivalent formulation as the It\^o SDE
\begin{equation}\label{eq:sdeStratoIto}
\left\lbrace
\begin{aligned}
\dd \mathbf{X}(t)&=\overline{G}(\mathbf{X}(t))\dd t+g(\mathbf{X}(t))\dd B(t),\quad t\ge 0,\\
\mathbf{X}(0)&=x_0,
\end{aligned}
\right.
\end{equation}
where the drift coefficient $\overline{G}$ is defined as
\begin{equation}\label{eq:Gbar}
\overline{G}=G+\frac12 g'g.
\end{equation}
Observe that if $G$ and $g$ satisfy the condition~\eqref{eq:Gg0}, then one has $\overline{G}(0)=0$, and one obtains the factorization
\begin{equation}
\overline{G}(x)=\overline{F}(x)x,\qquad \forall~x\in\R^+,
\end{equation}
where the mapping $\overline{F}\colon\R^+\to\R$ is defined as
\begin{equation}\label{eq:Fbar}
\overline{F}(x)=F(x)+\frac12 g'(x)f(x),\qquad \forall~x\in\R^+.
\end{equation}
Under Assumption~\ref{assCoeff}, the mapping $\overline{F}$ is continuous and bounded. Note that even if $G$ and $g$ satisfy Assumption~\ref{assCoeff}, it is necessary to impose stronger conditions on $g$ to ensure that $\overline{G}$ is of class $\mathcal{C}^2$ with bounded first and second order derivatives. It is sufficient for instance to assume in addition that $g$ is of class $\mathcal{C}^3$, and that $g$ and $g'''$ are bounded. When appropriate regularity and boundedness conditions hold for $\overline{G}$, the results of \Cref{propMomEx} also apply for the Stratonovich stochastic differential equation~\eqref{eq:sdeStrato}.

\section{First-order positivity preserving numerical scheme}\label{secPosi}

In this section, we present the construction of a new first-order positivity preserving scheme for the It\^o SDE~\eqref{eq:sde}, and for the Stratonovich SDE~\eqref{eq:sdeStrato}, see Section~\ref{sec:constructionscheme}. The almost sure preservation of positivity and the first-order convergence result are stated in Section~\ref{sec:main}. Finally, we illustrate our main results with numerical experiments in Section~\ref{sec:num}.

For an arbitrary time horizon $T\in(0,\infty)$ and a positive integer $N\in\N$, we define the time-step size by $\dt=T/N$.
For all $n\in\{0,\ldots,N-1\}$, set $t_n=n\dt$ for the time grid and define the Wiener increments by $\delta B_n=B(t_{n+1})-B(t_n)$.

\subsection{Construction of the numerical schemes}\label{sec:constructionscheme}

As a first step for the construction of the higher-order scheme for~\eqref{eq:sde}, it is convenient to recall first the construction of a basic numerical scheme with strong order $1/2$.

Let $n\in\{0,\ldots,N-1\}$, and assume that the approximation $X_n$ of the exact solution $X(t_n)$ at time $t_n$ has been constructed. Assume also that $X_n$ is a $\mathcal{F}_{t_n}$-measurable random variable. The value $X_{n+1}$ of the scheme at time $t_{n+1}=t_n+\dt$ is constructed by considering the equivalent formulation~\eqref{eq:sdef} of~\eqref{eq:sde} on the interval $[t_n,t_{n+1}]$, and by freezing the coefficients $F$ and $f$ at the numerical approximation
$X_n$: one then considers the auxiliary SDE
\begin{equation}\label{eq:auxsdescheme}
\left\lbrace
\begin{aligned}
\dd X_n(t)&=F(X_n)X_n(t)\dd t+f(X_n)X_n(t)\dd B(t),\quad t\in[t_n,t_{n+1}],\\
X_n(t_n)&=X_n,
\end{aligned}
\right.
\end{equation}
on the interval $[t_n,t_{n+1}]$, with initial value $X_n(t_n)=X_n$. The exact solution of the auxiliary SDE~\eqref{eq:auxsdescheme} is known: for all $t\in[t_n,t_{n+1}]$ one has
\[
X_n(t)=\exp\left(F(X_n)(t-t_n)+f(X_n)\bigl(B(t)-B(t_n)\bigr)-\frac12 f(X_n)^2(t-t_n)\right)X_n
.
\]
Setting $X_{n+1}=X_n(t_{n+1})$ provides the definition of the basic explicit numerical scheme: given the initial value $X_0=x_0$, for all $n\in\{0,\ldots,N-1\}$, set
\begin{equation}\label{eq:schemedemi}
X_{n+1}=\exp\left(F(X_n)\dt+f(X_n)\delta B_n-\frac12 f(X_n)^2\dt\right)X_n.
\end{equation}

Observe that the scheme~\eqref{eq:schemedemi} is positivity preserving: under the assumption that $X_0=x_0\ge 0$, for any value $\dt=T/N$ of the time-step size, almost surely $X_n\ge 0$ for all $n\in\{1,\ldots,N\}$. In addition, the scheme is consistent with the It\^o SDE~\eqref{eq:sde} and converges strongly with order $1/2$: using simplified version of the arguments developed below, one is able to show that for all $p\in[1,\infty)$ one has
\[
\underset{0\le n\le N}\sup~\bigl(\E[|X_n-X(t_n)|^p]\bigr)^{\frac1p}\lesssim\dt^{\frac12}.
\]
The details of the proof are not provided, since the objective of this article is to prove first-order strong convergence of a new positivity preserving numerical scheme. In the numerical experiments reported in Section~\ref{sec:num}, it will be verified that the strong order of convergence $1/2$ is observed and optimal for the scheme~\eqref{eq:schemedemi} applied to~\eqref{eq:sde} (when $F$ and $f$ are not constant).

A version of the scheme~\eqref{eq:schemedemi} for the Stratonovich SDE~\eqref{eq:sdeStrato} is  derived, considering its equivalent It\^o formulation: set $\mathbf{X}_0=x_0$, and for all $n\in\{0,\ldots,N-1\}$ set
\[
\mathbf{X}_{n+1}=\exp\left(\overline{F}(\mathbf{X}_n)\dt+f(\mathbf{X}_n)\delta B_n-\frac12 f(\mathbf{X}_n)^2\dt\right)\mathbf{X}_n.
\]
Recalling that $\overline{F}$ is given by~\eqref{eq:Fbar} and the identity $g'f=f^2+f'g$, the scheme is rewritten as
\begin{equation}\label{eq:schemedemiStrato}
\mathbf{X}_{n+1}=\exp\left(F(\mathbf{X}_n)\dt+f(\mathbf{X}_n)\delta B_n+\frac12(f'g)(\mathbf{X}_n)\dt\right)\mathbf{X}_n.
\end{equation}

\begin{remark}
The scheme~\eqref{eq:schemedemiStrato} is positivity preserving and provides an approximation at order $1/2$ of the solution to the Stratonovich SDE~\eqref{eq:sdeStrato}. Note that one cannot drop the term $\frac12(f'g)(\mathbf{X}_n)\dt$ in the exponential: the scheme
\[
\mathbf{X}_{n+1}=\exp\left(F(\mathbf{X}_n)\dt+f(\mathbf{X}_n)\delta B_n\right)\mathbf{X}_n
\]
associated with solving exactly auxiliary SDEs
\[
\left\lbrace
\begin{aligned}
\dd \mathbf{X}_n(t)&=F(\mathbf{X}_n)\mathbf{X}_n(t)\dd t+f(\mathbf{X}_n)\mathbf{X}_n(t)\circ\dd B(t),\quad t\in[t_n,t_{n+1}],\\
\mathbf{X}_n(t_n)&=\mathbf{X}_n,
\end{aligned}
\right.
\]
on each interval $[t_n,t_{n+1}]$ is not consistent with the Stratonovich SDE~\eqref{eq:sdeStrato}. This means that the factorization technique is well-adapted to the It\^o interpretation of the noise. When the noise is interpreted in the Stratonovich sense, one needs to consider the equivalent It\^o formulation to ensure consistency of the numerical scheme.
\end{remark}
The new first-order positivity preserving scheme can now be given. First, if one considers the Stratonovich SDE~\eqref{eq:sdeStrato}, substituting $\frac12(f'g)(\mathbf{X}_n)\dt$ by $\frac12(f'g)(\mathbf{X}_n)\delta B_n^2$ in the scheme \eqref{eq:schemedemiStrato}, one obtains the following scheme: set $\mathbf{X}_0=x_0$, and for all $n\in\{0,\ldots,N-1\}$ set
\begin{equation}\label{eq:schemeStrato}
\mathbf{X}_{n+1}=\exp\left(F(\mathbf{X}_n)\dt+f(\mathbf{X}_n)\delta B_n+\frac12(f'g)(\mathbf{X}_n)\delta B_n^2\right)\mathbf{X}_n.
\end{equation}
By construction, the new scheme~\eqref{eq:schemeStrato} is positivity preserving. Its consistency with the Stratonovich SDE~\eqref{eq:sdeStrato} is not obvious, it can be justified by Remark~\ref{rem:Stratoscheme} below, and is rigorously proved in the sequel.

\begin{remark}\label{rem:Stratoscheme}
To justify why the substitution above is relevant for reaching higher-order of convergence, assume that $G(x)=F(x)=0$ for all $x\in\R^+$. The scheme~\eqref{eq:schemeStrato} can be written as
\[
\mathbf{X}_{n+1}=\mathbf{\Phi}_{\delta B_n}(\mathbf{X}_n),
\]
where the mapping $(s,x)\in\R\times\R^+\mapsto\mathbf{\Phi}_s(x)$ is defined by
\[
\mathbf{\Phi}_s(x)=\exp\left(f(x)s+\frac12(f'g)(x)s^2\right)x.
\]
It is straightforward to check that for any fixed $x\in\R^+$, when $s\to 0$ one has
\begin{align*}
\mathbf{\Phi}_s(x)&=x+\bigl(f(x)x\bigr)s+\frac12\bigl(f(x)^2x+f'(x)g(x)x\bigr)s^2+{\rm O}(s^3)\\
&=x+g(x)s+\frac12(g'g)(x)s^2+{\rm O}(s^3)\\
&=\varphi_s(x)+{\rm O}(s^3),
\end{align*}
where $(s,x)\in\R\times\R^+\mapsto\varphi_s(x)$ is the flow associated to the ordinary differential equation
\[
\dot{x}=g(x).
\]
The mapping $\mathbf{\Phi}$ can be interpreted as a second-order integrator applied to that ordinary differential equation. The exact solution to the Stratonovich SDE~\eqref{eq:sdeStrato} when $G=0$ is given by
\[
\mathbf{X}(t)=\varphi_{B(t)}(x_0),\qquad \forall~t\ge 0.
\]
As a result, the scheme~\eqref{eq:schemeStrato} can naturally be expected to provide a first-order approximation for the Stratonovich SDE~\eqref{eq:sdeStrato}. In the sequel, we do not provide a direct proof of this result using the ideas described above, it will follow from the analysis performed for It\^o SDEs.
\end{remark}

Having constructed the new scheme~\eqref{eq:schemeStrato} for the Stratonovich SDE~\eqref{eq:sdeStrato}, we are now in the position to introduce it for the It\^o SDE~\eqref{eq:sde}, which admits the equivalent Stratonovich formulation
\[
\left\lbrace
\begin{aligned}
\dd X(t)&=\widehat{G}(X(t))\dd t+g(X(t))\circ \dd B(t),\quad t\ge 0,\\
X(0)&=x_0,
\end{aligned}
\right.
\]
with drift coefficient
\[
\widehat{G}=G-\frac12g'g.
\]
Observe that one has $\widehat{G}(0)=0$ and the factorization
\[
\widehat{G}(x)=\widehat{F}(x)x,\qquad \forall~x\in\R^+
\]
where the mapping $\widehat{F}\colon\R^+\to\R$ is defined as
\[
\widehat{F}(x)=F(x)-\frac12g'(x)f(x).
\]
Applying the new positivity preserving scheme~\eqref{eq:schemeStrato} for the  Stratonovich formulation above, one obtains the scheme defined by
\begin{align*}
{X}_{n+1}&=\exp\left(\widehat{F}({X}_n)\dt+f({X}_n)\delta B_n+\frac12(f'g)({X}_n)\delta B_n^2\right){X}_n\\
&=\exp\left(F(X_n)\dt+f(X_n)\delta B_n+\frac12\bigl(f' g\bigr)(X_n)\delta B_n^2-\frac12 \bigl(g'f\bigr)(X_n)\dt\right)X_n.
\end{align*}
Recalling again the identity $g'f=f^2+f'g$ finally yields the definition of the new first-order positivity preserving scheme for the It\^o SDE~\eqref{eq:sde}:
set $X_0=x_0$, and for all $n\in\{0,\ldots,N-1\}$, set
\begin{equation}\label{eq:scheme}
X_{n+1}=\exp\left(F(X_n)\dt+f(X_n)\delta B_n+\frac12\bigl(f' g\bigr)(X_n)\bigl[\delta B_n^2-\dt]-\frac12 f(X_n)^2\dt\right)X_n.
\end{equation}
Note that by construction the scheme~\eqref{eq:schemeStrato} for the Stratonovich SDE~\eqref{eq:sdeStrato} can be retrieved from the scheme~\eqref{eq:scheme} when it is applied to its equivalent It\^o formulation~\eqref{eq:sdeStratoIto}.

The properties of the scheme~\eqref{eq:scheme} are presented in Section~\ref{sec:main} and illustrated by numerical experiments in Section~\ref{sec:num}.

\begin{remark}
The explicit positivity preserving schemes~\eqref{eq:schemedemi} and~\eqref{eq:scheme} can respectively be written as
\begin{align*}
X_{n+1}&=\Phi\left(f(X_n)^2\dt,F(X_n)\dt+f(X_n)\delta B_n,X_n\right),\\
X_{n+1}&=\Phi\left(f(X_n)^2\dt,F(X_n)\dt+f(X_n)\delta B_n+\frac12\bigl(f' g\bigr)(X_n)\bigl[\delta B_n^2-\dt],X_n\right),
\end{align*}
where the mapping $\Phi\colon\R^+\times\R\times\R^+$ is defined by
\[
\Phi(s,\gamma,y)=\exp\left(\gamma-\frac{s}{2}\right)y.
\]
Observe that the mapping $\Phi$ provides the expression
\[
Y(t)=\Phi(t,B(t),y_0),\qquad \forall~t\ge 0,
\]
for the geometric Brownian motion which solves the auxiliary It\^o stochastic differential equation
\[
\dd Y(t)=Y(t)\dd B(t),\quad Y(0)=y_0.
\]
The schemes~\eqref{eq:schemedemi} and~\eqref{eq:scheme} thus have the same structure as the explicit domain preserving schemes introduced and studied in~\cite{DP}.
\end{remark}

\subsection{Main results}\label{sec:main}

In this section, we study properties of the scheme~\eqref{eq:scheme} applied to the It\^o SDE~\eqref{eq:sde}. Similar properties can be deduced for the scheme~\eqref{eq:schemeStrato} applied to the Stratonovich SDE~\eqref{eq:sdeStrato} under appropriate assumptions.

First, the scheme~\eqref{eq:scheme} is positivity preserving.
\begin{proposition}\label[proposition]{prop:posi}
Consider the SDE~\eqref{eq:sde} under Assumptions~\ref{assInit}~and~\ref{assCoeff}.
The numerical scheme~\eqref{eq:scheme} is positivity preserving: for any time $T\in(0,\infty)$ and any time-step size $\dt=T/N$ with $N\in\N$, almost surely, one has
\[
X_n\ge 0,\quad\forall n\in\{0,\ldots,N\}.
\]
\end{proposition}

\begin{proof}[Proof of \Cref{prop:posi}]
It suffices to apply a recursion argument on the index $n$.

First, if $n=0$, owing to Assumption~\ref{assInit} one has $X_0=x_0\geq0$.

Second, let $n\in\{0,\ldots,N-1\}$ and assume that almost surely one has $X_n\ge 0$. By construction of the numerical scheme~\eqref{eq:scheme}, one obtains the property $X_{n+1}\ge 0$ almost surely.

Completing the recursion argument then concludes the proof of \Cref{prop:posi}.
\end{proof}

It is straightforward to adapt \Cref{prop:posi} for the scheme~\eqref{eq:schemeStrato} applied to the Stratonovich SDE~\eqref{eq:sdeStrato}: the scheme~\eqref{eq:schemeStrato} is thus domain preserving.

The main result of this article, \Cref{theoMain}, states that the positivity preserving numerical scheme~\eqref{eq:scheme} converges strongly with order $1$.

\begin{theorem}\label[theorem]{theoMain}
Consider the It\^o SDE~\eqref{eq:sde} under Assumptions~\ref{assInit}~and~\ref{assCoeff}.

For all $T\in(0,\infty)$ and $p\in[1,\infty)$, there exists $C_p(T,x_0)\in(0,\infty)$ and $\tau_p\in(0,1)$, such that for any time-step size $\dt=T/N$, with $N\in\N$, which satisfies the condition $\dt\le \tau_p$, one has
\begin{equation}\label{eq:theoMain}
\underset{0\le n\le N}\sup~\bigl(\E\left[|X_n-X(t_n)|^{2p}\right]\bigr)^{\frac{1}{2p}}\le C_p(T,x_0)\dt.
\end{equation}
\end{theorem}
The proof of \Cref{theoMain} is given in \Cref{secProof}.

An immediate consequence of \Cref{theoMain} is the following result which provides strong and almost sure pathwise convergence with order $1-\varepsilon$ with arbitrary $\varepsilon\in(0,1)$.
\begin{corollary}\label{corCV}
Consider the It\^o SDE~\eqref{eq:sde} under Assumptions~\ref{assInit}~and~\ref{assCoeff}.

For all $\varepsilon\in(0,1)$, $T\in(0,\infty)$ and $p\in[1,\infty)$, there exists $C_{\varepsilon,p}(T,x_0)\in(0,\infty)$ and $\tau_p\in(0,1)$ such that for any time-step size $\dt=T/N$, with $N\in\N$, which satisfies the condition $\dt\le\tau_p$, one has
\begin{equation}\label{eq:corCV1}
\left(\E\bigl[\underset{0\le n\le N}\sup~|X_n-X(t_n)|^{2p}\bigr]\right)^{\frac{1}{2p}}\le C_{\varepsilon,p}(T,x_0)\dt^{1-\varepsilon}.
\end{equation}
Moreover, there exists an almost surely finite random variable $\mathcal{M}_{\varepsilon}(T,x_0)$ such that for any time-step size $\dt=T/N$, with $N\in\N$, one has
\begin{equation}\label{eq:corCV2}
\underset{0\le n\le N}\sup~|X_n-X(t_n)|\le \mathcal{M}_{\varepsilon}(T,x_0)\dt^{1-\varepsilon},\quad \forall~N\in\N,~{\rm a.s}.
\end{equation}
\end{corollary}

\begin{proof}[Proof of Corollary~\ref{corCV}]
Let $p\in[1,\infty)$ and $\varepsilon\in(0,1)$, and consider the auxiliary parameter $q=\max\bigl(p,\varepsilon^{-1}/2\bigr)$. One has the inequality
\[
\underset{0\le n\le N}\sup~|X_n-X(t_n)|^{2q}\le \sum_{n=1}^{N}|X_n-X(t_n)|^{2q}.
\]
Recalling that $\dt=T/N$ and applying the strong error estimate~\eqref{eq:theoMain} from \Cref{theoMain} (with exponent $2q$) one obtains the upper bound
\[
\E\bigl[\underset{0\le n\le N}\sup~|X_n-X(t_n)|^{2q}\bigr]\le \sum_{n=1}^{N}\E[|X_n-X(t_n)|^{2q}]\le C_q(T,x_0)\dt^{2q-1}.
\]
Finally, recalling that $q=\max\bigl(p,\varepsilon^{-1}/2\bigr)$, one has $p\le q$ and $\varepsilon>1/2q$, one obtains
\begin{align*}
\left(\E\bigl[\underset{0\le n\le N}\sup~|X_n-X(t_n)|^{2p}\bigr]\right)^{\frac{1}{2p}}&\le \left(\E\bigl[\underset{0\le n\le N}\sup~|X_n-X(t_n)|^{2q}\bigr]\right)^{\frac{1}{2q}}\\
&\le C_q(T,x_0)\dt^{1-\frac{1}{2q}}\\
&\le C_{\varepsilon,p}(T,x_0)\dt^{1-\varepsilon}.
\end{align*}
The proof of the pathwise strong error estimates~\eqref{eq:corCV1} is thus completed.

It remains to prove the almost sure error estimates~\eqref{eq:corCV2}. Let $\varepsilon\in(0,1)$ and consider the auxiliary parameter $p=2\varepsilon^{-1}$. Applying the Markov inequality and the pathwise strong error estimates~\eqref{eq:corCV1} (with rate $1-\varepsilon/2$), for all $\delta\in(0,1)$ one has
\begin{align*}
\sum_{N=1}^{\infty}\mathbb{P}\left(N^{1-\varepsilon}\underset{0\le n\le N}\sup~|X_n-X(t_n)|>\delta\right)&\le \delta^{-2p}\sum_{N=1}^{\infty}\E\left[N^{2p(1-\varepsilon)}\underset{0\le n\le N}\sup~|X_n-X(t_n)|^{2p}\right]\\
&\le C_{\varepsilon/2,p}(T,x_0)\delta^{-2p}\sum_{N=1}^{\infty}N^{2p(1-\varepsilon)}\dt^{2p(1-\frac{\varepsilon}{2})}\\
&\le C_{\varepsilon/2,p}(T,x_0)\delta^{-2p}\sum_{N=1}^{\infty}N^{-2},
\end{align*}
recalling that $\dt=T/N$ and that $p\varepsilon=2$.

The series above converges for all $\delta\in(0,1)$, therefore applying the Borel--Cantelli lemma, almost surely
\[
N^{1-\varepsilon}\underset{0\le n\le N}\sup~|X_n-X(t_n)|\underset{N\to\infty}\to 0.
\]
Then the almost sure error estimates~\eqref{eq:corCV2} is obtained where $\mathcal{M}_{\varepsilon}(T,x_0)$ is the almost surely finite random variable
\[
\mathcal{M}_{\varepsilon}(T,x_0)=T^{\varepsilon-1}\underset{N\in\N}\sup\left(N^{1-\varepsilon}\underset{0\le n\le N}\sup~|X_n-X(t_n)| \right).
\]
The proof of the almost sure error estimates~\eqref{eq:corCV2} is completed.

This concludes the proof of Corollary~\ref{corCV}.
\end{proof}

Another immediate consequence of \Cref{theoMain} is the first-order strong convergence of the scheme~\eqref{eq:schemeStrato} applied to the Stratonovich SDE~\eqref{eq:sdeStrato}, under appropriate assumptions on $G$ and $g$ to ensure that \Cref{assCoeff} is satisfied for the drift coefficient $\overline{G}=G+\frac12g'g$ given by~\eqref{eq:Gbar} and the diffusion coefficient $g$. As already mentioned, it suffices for instance to assume that the \Cref{assCoeff} is satisfied for $G$ and $g$, and to assume in addition that $g$ is of class $\mathcal{C}^3$, and that $g$ and $g'''$ are bounded.

\begin{corollary}\label{corStrato}
Consider the Stratonovich SDE~\eqref{eq:sde} under Assumptions~\ref{assInit}~and~\ref{assCoeff} for the coefficients $\overline{G}=G+\frac12g'g$ and $g$.

For all $T\in(0,\infty)$ and $p\in[1,\infty)$, there exists $C_p(T,x_0)\in(0,\infty)$ and $\tau_p\in(0,1)$, such that for any time-step size $\dt=T/N$, with $N\in\N$, which satisfies the condition $\dt\le \tau_p$, one has
\begin{equation}
\underset{0\le n\le N}\sup~\bigl(\E\left[|\mathbf{X}_n-\mathbf{X}(t_n)|^{2p}\right]\bigr)^{\frac{1}{2p}}\le C_p(T,x_0)\dt.
\end{equation}
\end{corollary}
A version of Corollary~\ref{corCV} can also be obtained for the scheme~\eqref{eq:schemeStrato} applied to the Stratonovich SDE~\eqref{eq:sdeStrato}, this is omitted.

\subsection{Numerical experiments}\label{sec:num}

In this subsection, we illustrate the main results of this article stated in Section~\ref{sec:main}.
 
For the It\^o SDE~\eqref{eq:sde}, we consider the first-order positivity preserving scheme~\eqref{eq:scheme} denoted by \hosI.  We illustrate its performance compared with the lower-order positivity preserving scheme~\eqref{eq:schemedemi} (denoted by \sI), and with the classical Euler--Maruyama scheme (denoted by {\upshape EM})
\[
X_{n+1}=X_{n}+G(X_n)\dt+g(X_n)\delta B_n,\quad n\in\{0,\ldots,N-1\},
\]
and the Milstein scheme (denoted by {\upshape Mil})
\[
X_{n+1}=X_{n}+G(X_n)\dt+g(X_n)\delta B_n+\frac12(g'g)(X_n)\bigl(\delta B_n^2-\dt\bigr),\quad n\in\{0,\ldots,N-1\}.
\]

We first illustrate \Cref{prop:posi}, with initial value $x_0=0.1$, $T=50$, and several choices of coefficients $G$ and $g$. We simulate $100$ realisations of the different numerical schemes and compute the proportion of realisations that remain positive at all times in Table~\ref{tablePosi}.
This table confirms the result stated in \Cref{prop:posi} thereby establishing a better performance of the numerical scheme~\eqref{eq:scheme} compared with the Euler--Maruyama and Milstein schemes.

\begin{table}[h]
\begin{center}
\begin{tabular}{|c | c |c | c| c| c|}
  \hline
  $G(x)$ & $g(x)$ &  \hosI & {\upshape Mil} & \sI & {\upshape EM} \\
  \specialrule{.1em}{.05em}{.05em}
  $x$ & $2.9\sin(x)$ & $100/100$ & $99/100$ & $100/100$ & $10/100$ \\
  \hline
  $x$ & $5\sin(x)$ & $100/100$ & $5/100$ & $100/100$ & $10/100$ \\
  \hline
  $x/(1+x^2)$ & $2\log(1+x)$ & $100/100$ & $56/100$ & $100/100$ & $2/100$ \\
  \hline
  $x$ & $4.3x$ & $100/100$ & $0/100$ & $100/100$ & $0/100$ \\
  \hline
\end{tabular}
\end{center}
\caption{Illustration of \Cref{prop:posi}: Proportion of positive realisations of the positivity preserving scheme~\eqref{eq:scheme}
(\hosI), the Milstein scheme ({\upshape Mil}), the positivity preserving scheme~\eqref{eq:schemedemi}
(\sI), and the Euler--Maruyama scheme ({\upshape EM}).}
\label{tablePosi}
\end{table}

A further illustration of \Cref{prop:posi} is given in Figure~\ref{figPosi}, where one realization for each numerical
scheme is displayed for the SDE~\eqref{eq:sde} with coefficients $G(x)=x$ and $g(x)=4.3x$.

% Subfigures 
\begin{figure}
\centering
\begin{subfigure}{.5\textwidth}
  \centering
  \includegraphics[width=1.\linewidth]{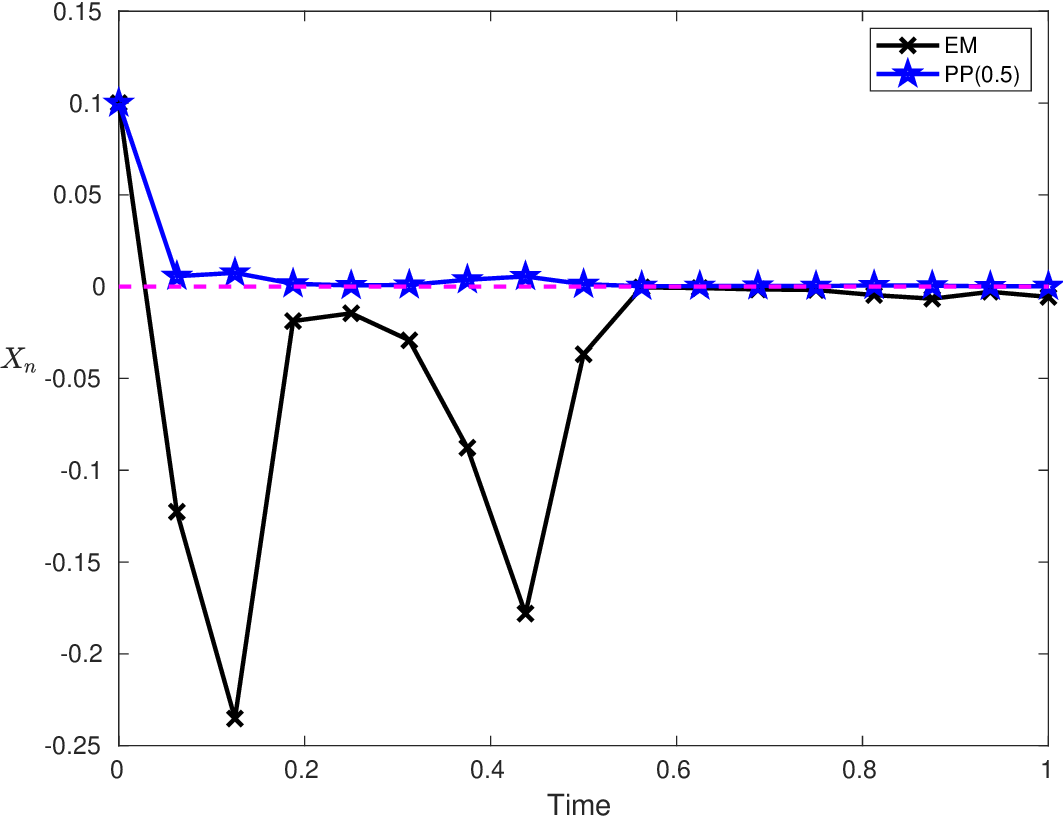}
  \caption{{\upshape EM} and \sI.}
  \label{figPosi1}
\end{subfigure}%
\begin{subfigure}{.5\textwidth}
  \centering
  \includegraphics[width=1.\linewidth]{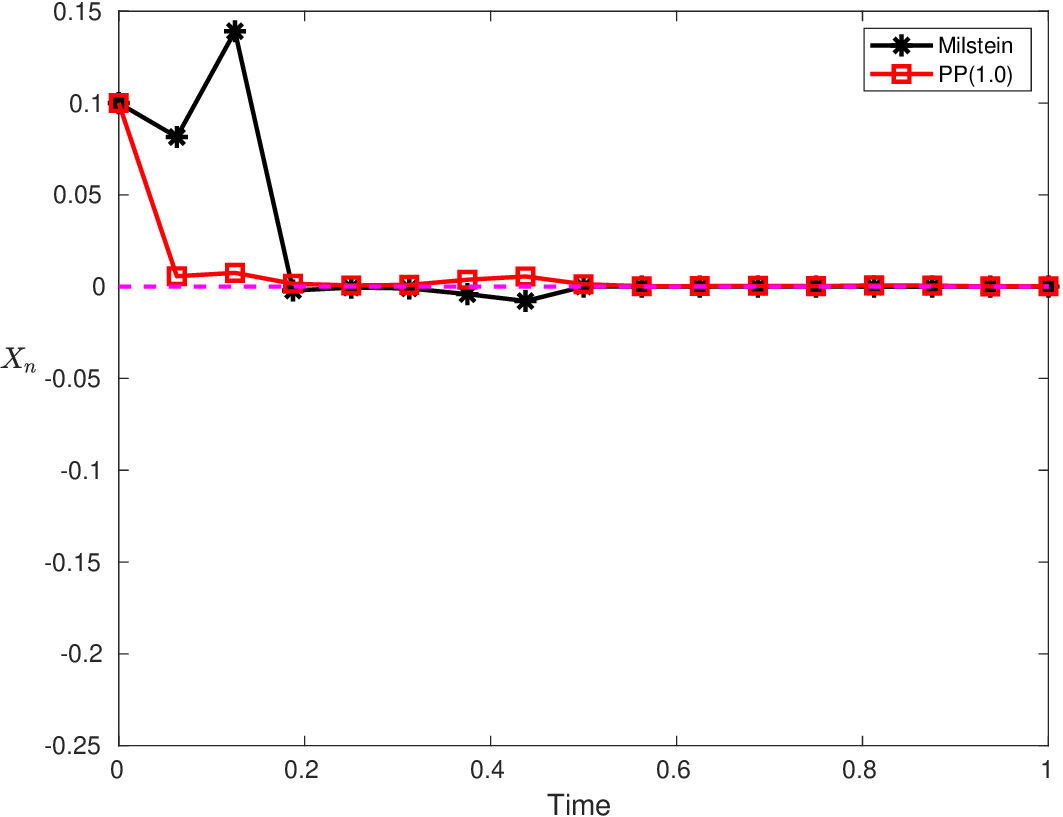}
  \caption{{\upshape Mil} and \hosI.}
  \label{figPosi2}
\end{subfigure}
\caption{Illustration of \Cref{prop:posi} for the SDE~\eqref{eq:sde} with $G(x)=x$ and $g(x)=4.3x$:
Evolution of the Euler--Maruyama scheme ({\upshape EM}), the Milstein scheme ({\upshape Mil})),
the positivity preserving scheme~\eqref{eq:schemedemi} (\sI) and the positivity preserving scheme~\eqref{eq:scheme} (\hosI).}
\label{figPosi}
\end{figure}

Next, we illustrate the first-order convergence rate of the proposed positivity preserving scheme~\eqref{eq:scheme} in the mean-square sense, i.e. when choosing $p=1$ in~\Cref{theoMain}.
For this purpose, we consider the It\^o SDE~\eqref{eq:sde} with the initial value $x_0=0.8$ and with the final time $T=1$.
The numerical schemes are applied with the time-step sizes $\dt=2^{-k}$ for $k\in\{6,7,\dots,14\}$.
The reference solution $X^{\text{ref}}$ is computed using the proposed first-order positivity preserving scheme~\eqref{eq:scheme} (\hosI) applied with the fine time-step size $\dt_{\text{ref}}=2^{-14}$.
The mean-square error is evaluated in the discrete supremum norm according to
\begin{equation*}
   \underset{0\leq n\leq N}\sup~\left( \E\left[|X_n-X_n^{\text{ref}}|^2\right] \right)^{1/2}.
\end{equation*}
The expectation is approximated by the Monte Carlo method averaging over $M_s=10^3$ independent sample paths.
This sample size was verified to be sufficient to have a negligible statistical error for the observation of the
strong rate of convergence. 

The results are reported in Figure~\ref{fig:ms}. In Figure~\ref{fig:msA} (left), the drift and diffusion coefficients are given by $G(x)=x$ and $g(x)=\sin(x)$. In Figure~\ref{fig:msB} (right), the drift and diffusion coefficients are given by $G(x)=\frac{x}{1+x^2}$, $g(x)=\ln(1+x)$. In both cases, \Cref{assInit} is satisfied. 
Figure~\ref{fig:ms} illustrates \Cref{theoMain}: we observe that the positivity preserving scheme \hosI~is a first-order method, like the Milstein scheme, whereas the positivity preserving scheme \sI~ and the Euler--Maruyama scheme converge with strong order $1/2$.

\begin{figure}[h]
\centering
\begin{subfigure}{.4\textwidth}
  \centering
  \includegraphics[width=\textwidth]{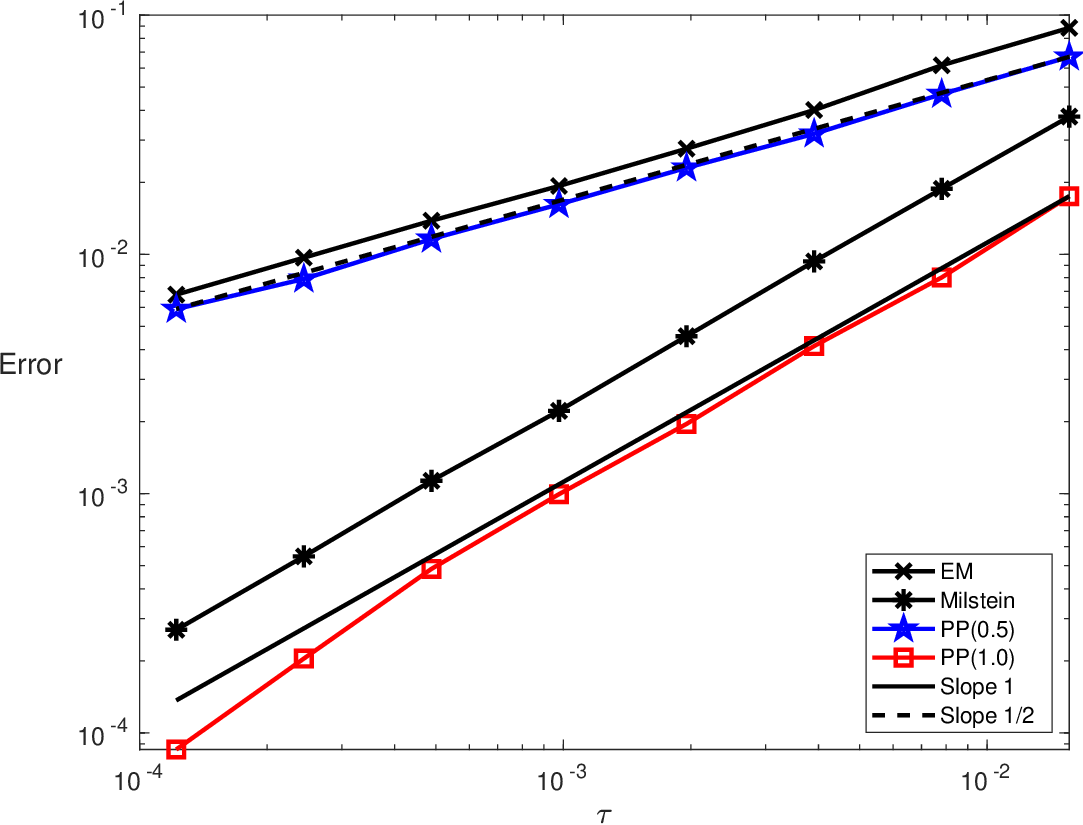}
  \caption{$G(x)=x$ and $g(x)=\sin(x)$.}\label{fig:msA}
\end{subfigure}%
\begin{subfigure}{.4\textwidth}
  \centering
  \includegraphics[width=\textwidth]{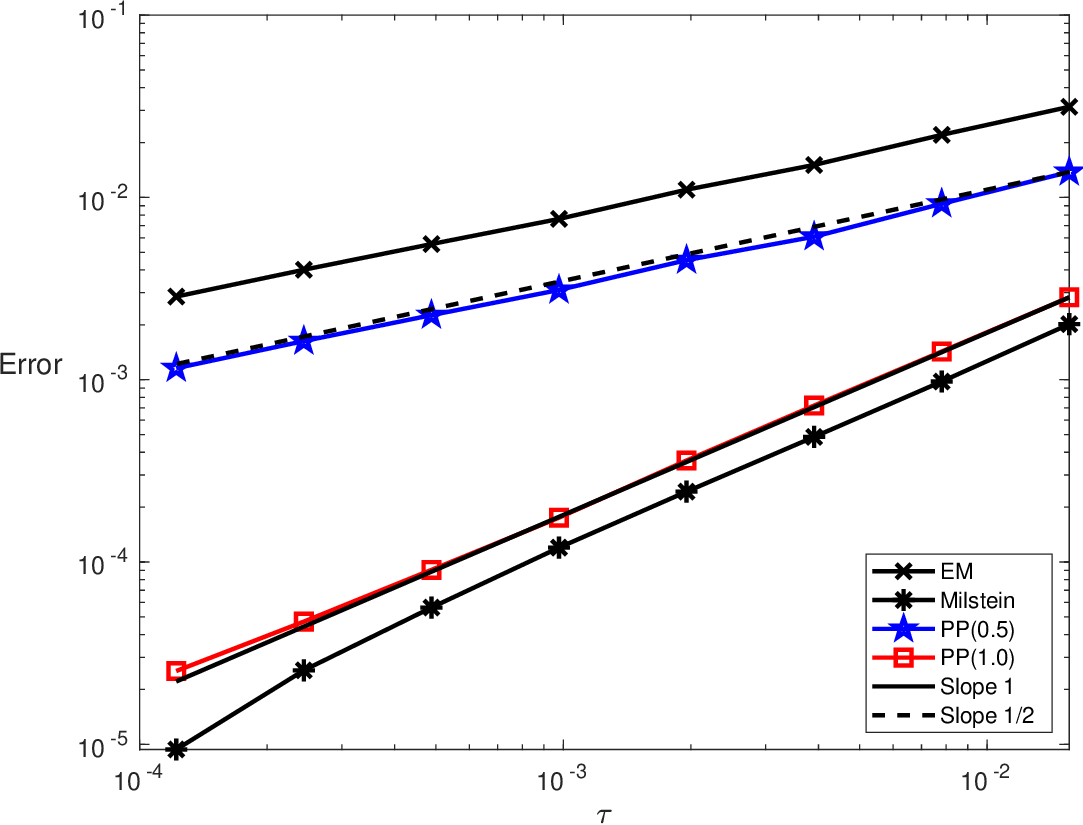}
  \caption{$G(x)=\frac{x}{1+x^2}$ and $g(x)=\ln(1+x)$. }\label{fig:msB}
\end{subfigure}%
\caption{Illustration of \Cref{theoMain}: mean-square convergence of the Euler--Maruyama scheme ({\upshape EM}), the Milstein scheme ({\upshape Mil}) and
the positivity preserving schemes~\eqref{eq:scheme}~and~\eqref{eq:schemedemi} (\hosI~and \sI) applied to the It\^o SDE~\eqref{eq:sde}.}
\label{fig:ms}
\end{figure}

Finally, the next numerical experiment illustrates \Cref{corStrato} on the convergence of the positivity preserving scheme~\eqref{eq:schemeStrato} applied to the Stratonovich SDE~\eqref{eq:sdeStrato}. The initial value is $x_0=1$, the final time is $T=1$, and the drift and diffusion coefficients are given by $G(x)=\sin(x)$ and $g(x)=\frac{x}{1+x^2}$. The time step-sizes $\dt=2^{-k}$ for $k\in\{6,7,\dots,14\}$ and the sample size $M_s=10^3$ for the Monte Carlo averaging are the same as above. The reference solution is computed using the proposed first-order positivity preserving scheme~\eqref{eq:schemeStrato} (\hosI) applied with the fine time-step size $\dt_{\text{ref}}=2^{-14}$. In Figure~\ref{fig:msStrato}, we consider the first-order positivity preserving scheme~\eqref{eq:schemeStrato} (\hosI), the positivity preserving scheme~\eqref{eq:schemedemiStrato} (\sI), the Euler--Heun scheme (denoted by {\upshape Heun})
\[
\left\lbrace
\begin{aligned}
&\widehat{\mathbf{X}}_{n+1}=\mathbf{X}_{n}+G(\mathbf{X}_n)\dt+g(\mathbf{X}_n)\delta B_n\\
&\mathbf{X}_{n+1}=\mathbf{X}_n+\frac12\left(G(\mathbf{X}_n)+G(\widehat{\mathbf{X}}_{n+1})\right)\dt+
\frac12\left(g(\mathbf{X}_n)+g(\widehat{\mathbf{X}}_{n+1})\right)\delta B_n
\end{aligned}
\right.
\]
and the Milstein scheme (denoted by {\upshape Mil})
\[
\mathbf{X}_{n+1}=\mathbf{X}_{n}+G(\mathbf{X}_n)\dt+g(\mathbf{X}_n)\delta B_n+\frac12(g'g)(\mathbf{X}_n)\delta B_n^2,\quad n\in\{0,\ldots,N-1\}.
\]

Figure~\ref{fig:msStrato} illustrates \Cref{corStrato}: we observe that the positivity preserving scheme \hosI~is a first-order method, like the Euler--Heun and Milstein schemes, whereas the positivity preserving scheme \sI~converge with strong order $1/2$.

\begin{figure}[h]
\centering
  \includegraphics[width=.5\textwidth]{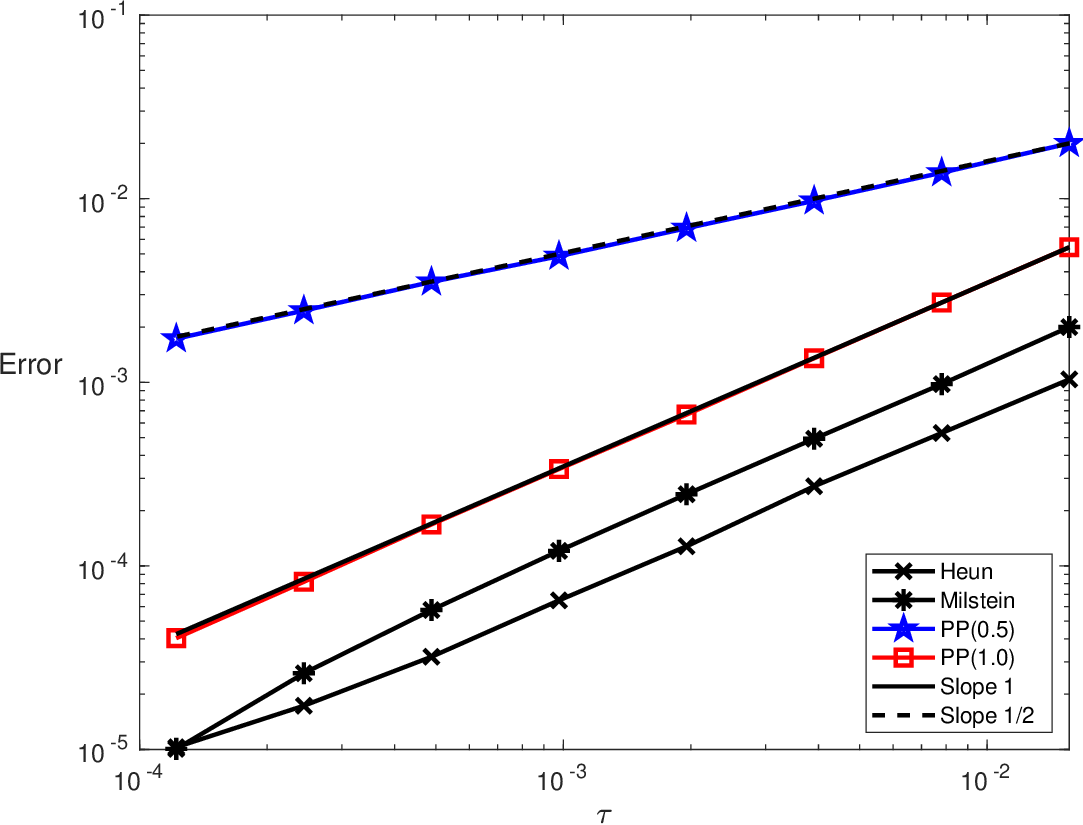}
\caption{Illustration of \Cref{corStrato}: mean-square convergence of the Euler--Heun scheme ({\upshape Heun}), the Milstein scheme ({\upshape Mil}) and
the positivity preserving schemes~\eqref{eq:schemeStrato}~and~\eqref{eq:schemedemiStrato} (\hosI~and \sI) applied to the Stratonovich SDE~\eqref{eq:sdeStrato}.}
\label{fig:msStrato}
\end{figure}

\section{Auxiliary results}\label{secAux}
In this section, we introduce an auxiliary continuous-time process (see~\eqref{eqAux} below) which coincides at the grid times $t_n=n\dt$ with the numerical solution given by~\eqref{eq:scheme}. \Cref{propMomNum} provides moment bounds for the numerical solution and the auxiliary process. Finally \Cref{propregExNum} provides moment bounds on the increments $X(t)-X(t_n)$ for the exact solution and the increments $\widetilde{X}(t)-\widetilde{X}(t_n)$ for the auxiliary process. 
These auxiliary results are applied in the proof of \Cref{theoMain} presented in \Cref{secProof}.

The auxiliary process $\bigl(\widetilde{X}(t)\bigr)_{t\in[0,T]}$ is defined as follows: for all $n\in\{0,\ldots,N-1\}$ and $t\in[t_n,t_{n+1})$, set
\begin{equation}\label{eqAux}
\left\lbrace
\begin{aligned}
\widetilde{X}(t)&=\exp\bigl(\widetilde{Y}(t)\bigr)X_n,\\
\widetilde{Y}(t)&=F(X_n)(t-t_n)+f(X_n)\bigl(B(t)-B(t_n)\bigr)\\
&+\frac12\bigl(f' g\bigr)(X_n)\bigl[(B(t)-B(t_n))^2-(t-t_n)\bigr]-\frac12 f(X_n)^2(t-t_n).
\end{aligned}
\right.
\end{equation}
Note that the auxiliary process $\bigl(\widetilde{X}(t)\bigr)_{t\in[0,T]}$ depends on the time-step size $\dt$
but this is omitted in the notation for simplicity. Observe that $\widetilde{X}_n=X_n$ for all $n\in\{0,\ldots,N\}$ owing to the definition~\eqref{eq:scheme} of the scheme, and that the process $\bigl(\widetilde{X}(t)\bigr)_{t\in[0,T]}$ has thus continuous trajectories on $[0,T]$. Moreover, under Assumptions~\ref{assInit}~and~\ref{assCoeff}, $\widetilde{X}(t)\ge 0$ for all $t\in[0,T]$ almost surely thanks to \Cref{prop:posi}.

We now show that the auxiliary process~\eqref{eqAux} is solution to a stochastic differential equation on each sub-interval $[t_n,t_{n+1}]$. First, observe that for all $t\in[t_n,t_{n+1}]$ one has almost surely
\[
(B(t)-B(t_n))^2-(t-t_n)=2\int_{t_n}^{t}\bigl[B(s)-B(t_n)\bigr]\dd B(s).
\]
Therefore, on each subinterval $[t_n,t_{n+1})$, the process $t\mapsto\widetilde{Y}(t)$ is solution to the stochastic differential equation
\[
\dd \widetilde{Y}(t)=\bigl[F(X_n)-\frac12f(X_n)^2\bigr]\dd t+\bigl[f(X_n)+\bigl(f' g\bigr)(X_n)\bigl(B(t)-B(t_n)\bigr)\bigr]\dd B(t).
\]
Second, applying It\^o's formula, for all $t\in[t_n,t_{n+1})$ one has
\begin{align*}
\dd \widetilde{X}(t)&=\widetilde{X}(t)\bigl[F(X_n)-\frac12 f(X_n)^2\bigr] \dd t \nonumber\\
&+\widetilde{X}(t)\bigl[f(X_n)+\bigl(f' g\bigr)(X_n)\bigl(B(t)-B(t_n)\bigr)\bigr]\dd B(t) \nonumber\\
&+\frac12 \widetilde{X}(t)\bigl[f(X_n)+\bigl(f' g\bigr)(X_n)\bigl(B(t)-B(t_n)\bigr)\bigr]^2 \dd t \nonumber,
\end{align*}
thus, on each sub-interval $[t_n,t_{n+1})$, the process $t\mapsto \widetilde{X}(t)$ is solution to the stochastic differential equation
\begin{align}
\dd \widetilde{X}(t)
&=\widetilde{X}(t)F(X_n)\dd t+\widetilde{X}(t)f(X_n)\dd B(t)\label{auxSDE}\\
&+\widetilde{X}(t)\bigl(f' g\bigr)(X_n)\bigl(B(t)-B(t_n)\bigr)\dd B(t) \nonumber\\
&+\widetilde{X}(t)\bigl(ff' g\bigr)(X_n)\bigl(B(t)-B(t_n)\bigr)\dd t\nonumber\\
&+\frac12\widetilde{X}(t)\bigl(f' g\bigr)^2(X_n)\bigl(B(t)-B(t_n)\bigr)^2\dd t,\nonumber
\end{align}
with initial value $\widetilde{X}(t_n)=X_n$.

We now provide moment bounds for the numerical solution $\bigl(X_n\bigr)_{0\le n\le N}$ to the scheme~\eqref{eq:scheme} and for the auxiliary process $\bigl(\widetilde{X}(t)\bigr)_{t\in[0,T]}$ defined by~\eqref{eqAux}.
\begin{proposition}\label[proposition]{propMomNum}
Let Assumptions~\ref{assInit} and~\ref{assCoeff} be satisfied.

For all $T\in(0,\infty)$ and $p\in[1,\infty)$, there exists $\dt_p\in(0,1)$ and $C_p(T,x_0)\in(0,\infty)$ such that
\begin{equation}\label{momentNum}
\underset{\dt\in(0,\dt_p)}\sup~\underset{0\le n\le N}\sup~\E\left[X_n^p\right]+\underset{\dt\in(0,\dt_p)}\sup~\underset{t\in[0,T]}\sup~\E\left[\widetilde{X}(t)^p\right]\le C_p(T,x_0).
\end{equation}
\end{proposition}

\begin{proof}[Proof of \Cref{propMomNum}]
Let us first note the following result: if $\gamma\sim\mathcal{N}(0,1)$ is a standard real-valued Gaussian random variable, for all $a\in\R$ and $b\in(-\infty,1/2)$, one has
\begin{equation}\label{eq:auxgaussian}
\E\left[\exp\bigl(a\gamma+b\gamma^2\bigr)\right]=\frac{1}{\sqrt{1-2b}}\exp\left(\frac{a^2}{2(1-2b)}\right).
\end{equation}

For all $p\in[1,\infty)$, define
\[
\dt_p=\frac{1}{1+2p\|f'g\|_\infty}.
\]

Let $n\in\{0,\ldots,N-1\}$. Owing to the tower property of conditional expectation, one has
\begin{align*}
\E\left[X_{n+1}^p\right]&=\E\left[\E\left[X_{n+1}^p|\mathcal{F}_{t_n}\right]\right]=\E\left[I_n\exp\Bigl(p\tau\bigl(F(X_n)-\frac12(g'f)(X_n)\Bigr)X_n^p\right]\\
&\le \E\left[I_n\exp\Bigl(p\bigl(\|F\|_\infty+\frac12\|g'f\|_\infty\bigr)\tau\Bigr)\right],
\end{align*}
where the random variable $I_n$ is defined as
\[
I_n=\E\left[\exp\Bigl(pf(X_n)\delta B_n+\frac{p}{2}(f'g)(X_n)\delta B_n^2 \Bigr)|\mathcal{F}_{t_n}\right].
\]
The Wiener increment $\delta B_n$ is independent of $\mathcal{F}_{t_n}$, therefore applying the identity~\eqref{eq:auxgaussian} with $a=pf(X_n)\sqrt{\dt}$ and $b=\frac{p}{2}(f'g)(X_n)\dt$, if $\dt\le \tau_p$ one has $b\le 1/4$, and one obtains almost surely
\[
I_n=\frac{1}{\sqrt{1-p(f'g)(X_n)\dt}}\exp\left(\frac{p^2f(X_n)^2\dt}{2(1-p(f'g)(X_n)\dt)}\right).
\]
In addition, if $\dt<\dt_p$ one has the lower bound $1-p(f'g)(X_n)\dt\ge 1/2$, and one thus obtains the upper bound
\[
I_n\le \frac{1}{\sqrt{1-p\|f'g\|_\infty\dt}}\exp\bigl(p^2\|f\|_\infty^2\dt\bigr).
\]
Finally, setting
\[
\mu=\underset{z\in(0,1/2]}\sup~\left(-\frac{\log(1-z)}{2z}\right)\in(0,\infty).
\]
one has
\[
\frac{1}{\sqrt{1-z}}\le e^{\mu z},\quad \forall~z\in(0,1/2].
\]

Set
\[
C_p=p\mu\|f'g\|_\infty+p^2\|f\|_\infty^2+p\bigl(\|F\|_\infty+\frac12\|g'f\|_\infty\bigr).
\]
Then, for all $\dt\in(0,\tau_p)$ one obtains the upper bound
\[
I_n\le \exp\Bigl(\bigl(p\mu\|f'g\|_\infty+p^2\|f\|_\infty^2\bigr)\dt\Bigr)
\]
and for all $n\in\{0,\ldots,N-1\}$ one has
\[
\E\left[X_{n+1}^p\right]\le \exp\bigl(C_p\dt\bigr)\E\left[X_n^p\right].
\]
Letting $C_p(T)=\exp(C_pT)$, it is then straightforward to obtain the upper bound
\[
\underset{0\le n\le N}\sup~\E\left[\widetilde{X}(t_n)^p\right]=\underset{0\le n\le N}\sup~\E\left[X_n^p\right]\le C_p(T)x_0^p\le C_p(T,x_0).
\]
This concludes the proof of the moment bounds for the numerical solution $\bigl(X_n\bigr)_{0\le n\le N}$. It remains to prove them for the auxiliary process $\bigl(\widetilde{X}(t)\bigr)_{t\in[0,T]}$. The arguments are similar to those applied above, they are provided for completeness.

Let $n\in\{0,\ldots,N-1\}$ and $t\in(t_n,t_{n+1})$. Owing to the tower property of conditional expectation, one has
\begin{align*}
\E\left[\widetilde{X}(t)^p\right]&=\E\left[\E\left[\widetilde{X}(t)^p|\mathcal{F}_{t_n}\right]\right]=\E\left[I_n(t)\exp\Bigl(p\bigl(F(X_n)-\frac12(g'f)(X_n)(t-t_n)\Bigr)X_n^p\right]\\
&\le \E\left[I_n(t)\exp\Bigl(p\bigl(\|F\|_\infty+\frac12\|g'f\|_\infty\bigr)\tau\Bigr)\right],
\end{align*}
where the random variable $I_n(t)$ is defined as
\[
I_n(t)=\E\left[\exp\Bigl(pf(X_n)\bigl(B(t)-B(t_n)\bigr)+\frac{p}{2}(f'g)(X_n)\bigl(B(t)-B(t_n)\bigr)^2 \Bigr)|\mathcal{F}_{t_n}\right].
\]
The Wiener increment $B(t)-B(t_n)$ is independent of $\mathcal{F}_{t_n}$, therefore applying the identity~\eqref{eq:auxgaussian} with $a=pf(X_n)\sqrt{t-t_n}$ and $b=\frac{p}{2}(f'g)(X_n)(t-t_n)$, if $t-t_n\le \dt\le \tau_p$ one has $b\le 1/4$, and one obtains almost surely
\begin{align*}
I_n(t)&=\frac{1}{\sqrt{1-p(f'g)(X_n)(t-t_n)}}\exp\left(\frac{p^2f(X_n)^2(t-t_n)}{2(1-p(f'g)(X_n)(t-t_n))}\right)\\
&\le \exp\left(\mu p(f'g)(X_n)(t-t_n)\right)\exp\left(p^2f(X_n)^2(t-t_n)\right)\\
&\le \exp\left(\bigl(p\mu\|f'g\|_\infty+p^2\|f\|_\infty^2\bigr)\tau\right).
\end{align*}
Therefore, for all $n\in\{0,\ldots,N-1\}$ and all $t\in(t_n,t_{n+1})$ one has
\[
\E\left[\widetilde{X}(t)^p\right]\le \exp(C_p\dt)\E[X_n^p].
\]
Applying the moment bounds obtained previously for the numerical solution then provides the moment bounds for the auxiliary process.

The proof of \Cref{propMomNum} is thus completed.
\end{proof}

To conclude this section, we provide moment bounds on the increments of the exact and numerical solution.

\begin{proposition}\label[proposition]{propregExNum}
Let Assumptions~\ref{assInit}~and~\ref{assCoeff} be satisfied.

Let $\bigl(X(t)\bigr)_{t\in[0,T]}$ be the solution to the It\^o SDE~\eqref{eq:sde} and $\bigl(\widetilde{X}(t)\bigr)_{t\in[0,T]}$ be the auxiliary process defined by~\eqref{eqAux}.

For all $T\in(0,\infty)$, and any $p\in[1,\infty)$, there exists $C_p(T,x_0)\in(0,\infty)$ and $\dt_p\in(0,1)$, such that for all $N\in\N$, if the time-step size satisfies $\dt=T/N\le \tau_p$ for all $n\in\{0,\ldots,N-1\}$ and all $t\in([)t_n,t_{n+1})$ one has
\begin{equation}\label{regExNum}
\E\left[|{X}(t)-{X}(t_n)|^p\right]+\E\left[|\widetilde{X}(t)-\widetilde{X}(t_n)|^p\right]\le C_p(T,x_0)\dt^{\frac{p}{2}}.
\end{equation}
\end{proposition}

\begin{proof}[Proof of \Cref{propregExNum}]
First, let us prove the temporal regularity property for the exact solution $X(t)$.
Let $n\in\{0,\ldots,N-1\}$ and $t\in(t_n,t_{n+1})$. Writing the integral formulation of the solution, and applying the 
H\"older inequality and the BDG inequality~\eqref{eq:BDG}, one has
\begin{align*}
\E\left[|{X}(t)-{X}(t_n)|^{2p}\right]&\le C_p\E\left[ \left| \int_{t_n}^t G(X(s))\dd s\right|^{2p}\right]+C_p\E\left[\left|\int_{t_n}^t g(X(s)))\dd B(s) \right|^{2p} \right]\\
&\leq C_p(T)(t-t_n)^{2p-1}\int_{t_n}^t\E\left[|G(X(s))|^{2p}\right]\dd s \\
&+C_p(T)(t-t_n)^{p-1}\int_{t_n}^t\E\left[|g(X(s))|^{2p}\right]\dd s.
\end{align*}
Applying the property~\eqref{eq:Gg-growth} of the coefficients $G$ and $g$, and the moment bounds~\eqref{momentExact} from \Cref{propMomEx}, one then obtains
\begin{align*}
\E\left[|{X}(t)-{X}(t_n)|^{2p}\right]
&\le C_p(T)\left((t-t_n)^{2p}\norm{G'}^{2p}_{\infty} +(t-t_n)^{p}\norm{g'}^{2p}_{\infty}\right)\underset{s\in[0,T]}\sup~\E\left[|X(s)|^{2p}\right]\\
&\le C_p(T,x_0)(t-t_n)^p\\
&\le C_p(T,x_0)\dt^p.
\end{align*}
The proof of the temporal regularity property~\eqref{regExNum} is thus completed for the exact solution. It remains to establish the result for the auxiliary process $\widetilde{X}$.

Let $n\in\{0,\ldots,N-1\}$ and $t\in[t_n,t_{n+1})$. On the interval $[t_n,t_{n+1}]$, the auxiliary process $\widetilde{X}$ is solution to the stochastic differential equation~\eqref{auxSDE}, therefore one obtains
\begin{align*}
\E\left[|\widetilde{X}(t)-\widetilde{X}(t_n)|^{2p}\right]&\le C_p\E\left[ \left|\int_{t_n}^t \widetilde{X}(s)F(X_n)\dd s\right|^{2p}\right]\\
&+C_p\E\left[\left|\int_{t_n}^t \widetilde{X}(s)f(X_n)\dd B(s)\right|^{2p}\right] \\
&+C_p\E\left[ \left|\int_{t_n}^t \widetilde{X}(s)(f'g)(X_n)(B(s)-B(t_n))\dd B(s)\right|^{2p}\right]\\
&+C_p\E\left[ \left|\int_{t_n}^t \widetilde{X}(s)(ff'g)(X_n)(B(s)-B(t_n))\dd s\right|^{2p}\right]\\
&+C_p\E\left[ \left|\int_{t_n}^t \widetilde{X}(s)(f'g)(X_n)^2(B(s)-B(t_n))^2\dd s\right|^{2p}\right].
\end{align*}
The mappings $F$, $f$ and $f'g$ are bounded, see~\eqref{eq:boundF},~\eqref{eq:boundf} and~\eqref{eq:boundf'g} in Section~\ref{secAss}. Applying the H\"older inequality, the BDG inequality~\eqref{eq:BDG}, and the moment bounds~\eqref{momentNum} from \Cref{propMomNum} on the auxiliary process $\widetilde{X}$ (which requires to impose a condition $\dt\le \tau_p$ on the time-step size), one obtains the upper bounds
\begin{align*}
\E\left[|\widetilde{X}(t)-\widetilde{X}(t_n)|^{2p}\right]&\le C_p(T)(t-t_n)^{2p-1}\int_{t_n}^t\E[|\widetilde{X}(s)F(X_n)|^{2p}]\dd s\\
&+C_p(T)(t-t_n)^{p-1}\int_{t_n}^t \E[|\widetilde{X}(s)f(X_n)|^{2p}]\dd s \\
&+C_p(T)(t-t_n)^{p-1}\int_{t_n}^t \E[|\widetilde{X}(s)(f'g)(X_n)(B(s)-B(t_n))|^{2p}]\dd s\\
&+C_p(T)(t-t_n)^{2p-1}\int_{t_n}^t \E[|\widetilde{X}(s)(ff'g)(X_n)(B(s)-B(t_n))|^{2p}]\dd s\\
&+C_p(T)(t-t_n)^{2p-1}\int_{t_n}^t \E[|\widetilde{X}(s)(f'g)(X_n)^2(B(s)-B(t_n))^2|^{2p}]\dd s\\
&\le C_p(T)(t-t_n)^{2p}\|F\|_\infty^{2p}\underset{s\in[0,T]}\sup~\E[|\widetilde{X}(s)|^{2p}]\\
&+C_p(T)(t-t_n)^{p}\|f\|_\infty^{2p}\underset{s\in[0,T]}\sup~\E[|\widetilde{X}(s)|^{2p}]\\
&+C_p(T)(t-t_n)^{2p}\|f'g\|_\infty^{2p}\underset{s\in[0,T]}\sup~\bigl(\E[|\widetilde{X}(s)|^{4p}]\bigr)^{\frac12}\\
&+C_p(T)(t-t_n)^{3p}\|f\|_\infty^{2p}\|f'g\|_\infty^{2p}\underset{s\in[0,T]}\sup~\bigl(\E[|\widetilde{X}(s)|^{2p}]\bigr)^{\frac12}\\
&+C_p(T)(t-t_n)^{4p}\|f'g\|_\infty^{4p}\underset{s\in[0,T]}\sup~\bigl(\E[|\widetilde{X}(s)|^{2p}]\bigr)^{\frac12}\\
&\le C_p(T,x_0)(t-t_n)^{p}\\
&\le C_p(T,x_0)\dt^{p}.
\end{align*}
The proof of the temporal regularity property~\eqref{regExNum} is thus completed for the auxiliary process.

This concludes the proof of \Cref{regExNum}.
\end{proof}

\section{Proof of \Cref{theoMain}}\label{secProof}

This section is devoted to the proof of the first-order convergence of the positivity preserving scheme~\eqref{eq:scheme} to the solution of the It\^o SDE~\eqref{eq:sde} stated in \Cref{theoMain}.
We first provide a decomposition of the error of the numerical scheme in \Cref{secDec}. \Cref{secAuxErr} provides auxiliary error bounds.
Finally \Cref{secOrder} provides the proof of strong convergence of the scheme, first at order $1/2$ and then at order $1$.

\subsection{Decomposition of the error}\label{secDec}
The first step in the proof of \Cref{theoMain} is to provide a decomposition of the error
\[
e_n=X_n-X(t_n)
\]
at time $t_n$, for all $n\in\{1,\ldots,N\}$,
between the exact solution $X(t_n)$ to the It\^o SDE~\eqref{eq:sde} and the numerical solution $X_n$ given by~\eqref{eq:scheme}.

One the one hand, consider the exact solution. For all $n\in\{1,\ldots,N\}$, one has
\begin{align*}
X(t_n)-x_0&=\int_{0}^{t_n}G(X(t))\dd t+\int_{0}^{t_n}g(X(t))\dd B(t)\\
&=\sum_{k=0}^{n-1}\int_{t_k}^{t_{k+1}}G(X(t))\dd t+\sum_{k=0}^{n-1}\int_{t_k}^{t_{k+1}}g(X(t))\dd B(t).
\end{align*}
Define the auxiliary error terms $\varepsilon_n^{(1,1)}$ and $\varepsilon_n^{(1,2)}$ as
\begin{align}
\varepsilon_n^{(1,1)}&=\sum_{k=0}^{n-1}\int_{t_k}^{t_{k+1}}\Bigl(G(X(t))-G(X(t_k))\Bigr)\dd t\label{err11}\\
\varepsilon_n^{(1,2)}&=\sum_{k=0}^{n-1}\int_{t_k}^{t_{k+1}}\Bigl(g(X(t))-g(X(t_k))-\bigl(g'g\bigr)(X(t_k))\bigl[B(t)-B(t_k)\bigr]\Bigr)\dd B(t)\label{err12}.
\end{align}
One then obtains for all $n\in\{1,\ldots,N\}$
\begin{align}\label{eq:decompEx}
X(t_n)-x_0&=\sum_{k=0}^{n-1}\int_{t_k}^{t_{k+1}}G(X(t_k))\dd t+\varepsilon_n^{(1,1)}\\
&+\sum_{k=0}^{n-1}\int_{t_k}^{t_{k+1}}\Bigl(g(X(t_k))+\bigl(g'g\bigr)(X(t_k))\bigl[B(t)-B(t_k)\bigr]\Bigr)\dd B(t)+\varepsilon_n^{(1,2)}.\nonumber
\end{align}
On the other hand, consider the numerical scheme~\eqref{eq:scheme} and the associated auxiliary process $\widetilde{X}$ defined by~\eqref{eqAux}. Applying a telescoping sum argument and recalling that on each interval $[t_k,t_{k+1}]$ the auxiliary process $\widetilde{X}$ is the solution to the auxiliary SDE~\eqref{auxSDE},
for all $n\in\{1,\ldots,N\}$, one has
\begin{align*}
X_n-x_0&=\sum_{k=0}^{n-1}\Bigl(X_{k+1}-X_{k}\Bigr)\\
&=\sum_{k=0}^{n-1}\Bigl(\widetilde{X}(t_{k+1})-\widetilde{X}(t_k)\Bigr)\\
&=\sum_{k=0}^{n-1}\int_{t_k}^{t_{k+1}}F(X_k)\widetilde{X}(t)\dd t+\sum_{k=0}^{n-1}\int_{t_k}^{t_{k+1}}f(X_k)\widetilde{X}(t)\dd B(t)\\
&+\sum_{k=0}^{n-1}\int_{t_k}^{t_{k+1}}\bigl(f'g\bigr)(X_k)\widetilde{X}(t)\bigl[B(t)-B(t_k)\bigr]\dd B(t)\\
&+\sum_{k=0}^{n-1}\int_{t_k}^{t_{k+1}}\bigl(ff'g\bigr)(X_k)\widetilde{X}(t)\bigl[B(t)-B(t_k)\bigr]\dd t\\
&+\frac12\sum_{k=0}^{n-1}\int_{t_k}^{t_{k+1}}\bigl(f'g\bigr)^2(X_k)\widetilde{X}(t)\bigl[B(t)-B(t_k)\bigr]^2\dd t.
\end{align*}
Define the auxiliary error terms $\varepsilon_n^{(2,j)}$ with $j=1,\ldots,5$ as
\begin{align}
\varepsilon_n^{(2,1)}&=\sum_{k=0}^{n-1}\int_{t_k}^{t_{k+1}}F(X_k)\Bigl(\widetilde{X}(t)-\widetilde{X}(t_k)\Bigr)\dd t\label{err21}\\
\varepsilon_n^{(2,2)}&=\sum_{k=0}^{n-1}\int_{t_k}^{t_{k+1}}f(X_k)\Bigl(\widetilde{X}(t)-\widetilde{X}(t_k)-g(X_k)\bigl[B(t)-B(t_k)\bigr]\Bigr)\dd B(t)\label{err22}\\
\varepsilon_n^{(2,3)}&=\sum_{k=0}^{n-1}\int_{t_k}^{t_{k+1}}\bigl(f'g\bigr)(X_k)\Bigl(\widetilde{X}(t)-\widetilde{X}(t_k)\Bigr)\bigl[B(t)-B(t_k)\bigr]\dd B(t)\label{err23}\\
\varepsilon_n^{(2,4)}&=\sum_{k=0}^{n-1}\int_{t_k}^{t_{k+1}}\bigl(ff'g\bigr)(X_k)\widetilde{X}(t)\bigl[B(t)-B(t_k)\bigr]\dd t\label{err24}\\
\varepsilon_n^{(2,5)}&=\frac12\sum_{k=0}^{n-1}\int_{t_k}^{t_{k+1}}\bigl(f'g\bigr)^2(X_k)\widetilde{X}(t)\bigl[B(t)-B(t_k)\bigr]^2\dd t.\label{err25}
\end{align}
One then obtains for all $n\in\{1,\ldots,N\}$
\begin{align}\label{eq:decompNum}
X_n-x_0&=\sum_{k=0}^{n-1}\int_{t_k}^{t_{k+1}}G(X_k)\dd t+\varepsilon_n^{(2,1)}\\
&+\sum_{k=0}^{n-1}\int_{t_k}^{t_{k+1}}\Bigl(g(X_k)+\bigl(fg\bigr)(X_k)\bigl[B(t)-B(t_k)\bigr]\Bigr)\dd B(t)+\varepsilon_{n}^{(2,2)}\nonumber\\
&+\sum_{k=0}^{n-1}\int_{t_k}^{t_{k+1}}\bigl(f'g\bigr)(X_k)X_k\bigl[B(t)-B(t_k)\bigr]\dd B(t)+\varepsilon_n^{(2,3)}\nonumber\\
&+\varepsilon_n^{(2,4)}+\varepsilon_n^{(2,5)}.\nonumber
\end{align}

Finally, for all $n\in\{1,\ldots,N\}$, set
\begin{align}
\varepsilon_{n}^{(3,1)}&=\tau\sum_{k=0}^{n-1}\Bigl(G(X_k)-G(X(t_k))\Bigr)\label{err31}\\
\varepsilon_{n}^{(3,2)}&=\sum_{k=0}^{n-1}\int_{t_k}^{t_{k+1}}\Bigl(g(X_k)-g(X(t_k))\Bigr)\dd B(t)\label{err32}\\
\varepsilon_{n}^{(3,3)}&=\sum_{k=0}^{n-1}\int_{t_k}^{t_{k+1}}\Bigl(\bigl(g'g)(X_k)-\bigl(g'g\bigr)(X(t_k))\Bigr)\bigl[B(t)-B(t_k)\bigr]\dd B(t).\label{err33}
\end{align}
Note that for all $x\ge 0$ one has
\[
g'(x)g(x)=f(x)g(x)+f'(x)g(x)x.
\]
Then combining the decompositions~\eqref{eq:decompEx} of $X(t_n)-x_0$ and~\eqref{eq:decompNum} of $X_n-x_0$, the error $e_n=X_n-X(t_n)$ is decomposed as follows: for all $n\in\{1,\ldots,N\}$, one has
\begin{equation}\label{eqDecErr}
e_n=-\sum_{j=1}^{2}\varepsilon_n^{(1,j)}+\sum_{j=1}^{5}\varepsilon_n^{(2,j)}+\sum_{j=1}^{3}\varepsilon_n^{(3,j)}.
\end{equation}

Upper bounds for moments of the error terms $\varepsilon_n^{(1,j)}$ ($j=1,2$), resp. $\varepsilon_n^{(2,j)}$ ($j=1,\ldots,5$), are provided in Lemma~\ref{lem:aux1}, resp.~\ref{lem:aux2}, stated and proved in Section~\ref{secAuxErr}. The error terms $\varepsilon_n^{(3,j)}$ ($j=1,2,3$) are treated in Section~\ref{secOrder}, where \Cref{theoMain} is finally established.

\subsection{Auxiliary error bounds}\label{secAuxErr}

In this section, we state and prove Lemmas~\ref{lem:aux1} and~\ref{lem:aux2} on the error terms $\varepsilon_n^{(1,j)}$ and $\varepsilon_n^{(2,j)}$ respectively.
\begin{lemma}\label{lem:aux1}
Under the setting of \Cref{theoMain}, one has the following upper bounds for the error terms defined by~\eqref{err11}-\eqref{err12}: for all $n\in\{1,\ldots,N\}$ one has
\begin{equation}\label{eqErr1}
\E\left[|\varepsilon_n^{(1,j)}|^{2p}\right]\leq C_p(T,x_0)\dt^{2p},\qquad \forall~j=1,2.
\end{equation}
\end{lemma}

\begin{proof}[Proof of Lemma~\ref{lem:aux1}]
The inequality~\eqref{eqErr1} is proved separately for the error terms $\varepsilon_{n}^{(1,1)}$ and $\varepsilon_{n}^{(1,2)}$.

{\bf Treatment of the error term $\varepsilon_{n}^{(1,1)}$.}

Recall the definition~\eqref{err11} of $\varepsilon_n^{(1,1)}$. For all $k\in\{0,\ldots,N-1\}$ and all $t\in(t_k,t_{k+1})$, applying It\^o's formula, one has
\[
G(X(t))-G(X(t_k))=\int_{t_k}^{t}\bigl(G'G+\frac12 G''g^2\bigr)(X(s))\dd s+\int_{t_k}^{t}(G'g)(X(s))\dd B(s).
\]
Integrating on the interval $[t_k,t_{k+1}]$ and applying deterministic and stochastic versions of the Fubini theorem, one thus obtains
\begin{align*}
\int_{t_k}^{t_{k+1}}\Bigl(G(X(t))-G(X(t_k))\Bigr)\dd t&=\int_{t_k}^{t_{k+1}}\bigl(t_{k+1}-s\bigr)\bigl(G'G+\frac12 G''g^2\bigr)(X(s))\dd s\\
&+\int_{t_k}^{t_{k+1}}\bigl(t_{k+1}-s\bigr)(G'g)(X(s))\dd B(s).
\end{align*}
Owing to \Cref{assCoeff}, the mappings $G'$ and $G''$ are bounded, and for the mapping $g$ one has the inequality~\eqref{eq:Gg-growth}. As a result, applying the moment bounds~\eqref{momentExact} from \Cref{propMomEx}, one obtains for all $s\in[0,T]$
\begin{align*}
\E\bigl[\bigl|\bigl(G'G+\frac12 G''g^2\bigr)(X(s))\bigr|^{2p}\bigr]+\E\left[\big|(G'g)(X(s))|^{2p}\right]&\le C_p\E[|X(s)|^{2p}]+ C_p\E[|X(s)|^{4p}]\\
&\le C_p(T,x_0).
\end{align*}
Applying the H\"older inequality and the discrete-time version~\eqref{eq:BDGdiscrete} of the BDG inequality, and the moment bounds above, one thus obtains the upper bounds
\begin{align*}
\E\left[|\varepsilon_n^{(1,1)}|^{2p}\right]&\le C_p(T)
\sum_{k=0}^{n-1}\int_{t_k}^{t_{k+1}}|t_{k+1}-s|^{2p}\E\bigl[\bigl|\bigl(G'G+\frac12 G''g^2\bigr)(X(s))\bigr|^{2p}\bigr]\dd s\\
&+C_p(T)\sum_{k=0}^{n-1}\int_{t_k}^{t_{k+1}}|t_{k+1}-s|^{2p}\E\left[\big|(G'g)(X(s))|^{2p}\right]\dd s\\
&\le C_p(T,x_0)\tau^{2p}.
\end{align*}
This concludes the treatment of the error term $\varepsilon_n^{(1,1)}$.

{\bf Treatment of the error term $\varepsilon_{n}^{(1,2)}$.}

Recall the definition~\eqref{err12} of $\varepsilon_n^{(1,2)}$. For all $k\in\{0,\ldots,N-1\}$ and all $t\in(t_k,t_{k+1})$, using It\^o's formula, one has
\begin{align*}
g(X(t))-g(X(t_k))&=\int_{t_k}^{t}\bigl(g'G+\frac12 g''g^2\bigr)(X(s))\dd s+\int_{t_k}^{t}(g'g)(X(s))\dd B(s)\\
&=(g'g)(X(t_k))\bigl[B(t)-B(t_k)\bigr]\\
&+\int_{t_k}^{t}\bigl(g'G+\frac12 g''g^2\bigr)(X(s))\dd s\\
&+\int_{t_k}^{t}\bigl[(g'g)(X(s))-(g'g)(X(t_k))\bigr]\dd B(s).
\end{align*}
Integrating on the interval $[t_k,t_{k+1}]$, one thus obtains
\begin{align*}
&\int_{t_k}^{t_{k+1}}\Bigl(g(X(t))-g(X(t_k))-\bigl(g'g\bigr)(X(t_k))\bigl[B(t)-B(t_k)\bigr]\Bigr)\dd B(t)\\
&=\int_{t_k}^{t_{k+1}}\int_{t_k}^{t}\bigl(g'G+\frac12 g''g^2\bigr)(X(s))\dd s\dd B(t)\\
&+\int_{t_k}^{t_{k+1}}\int_{t_k}^{t}\bigl[(g'g)(X(s))-(g'g)(X(t_k))\bigr]\dd B(s)\dd B(t).
\end{align*}
Applying the discrete-time version~\eqref{eq:BDGdiscrete} of the BDG inequality, one obtains the upper bound
\begin{align*}
\E\left[|\varepsilon_n^{(1,2)}|^{2p}\right]&\le C_p(T)\sum_{k=0}^{n-1}\int_{t_k}^{t_{k+1}}\E\left[\left|\int_{t_k}^{t}\bigl(g'G+\frac12 g''g^2\bigr)(X(s))\dd s\right|^{2p}\right]\dd t\\
&+C_p(T)\sum_{k=0}^{n-1}\int_{t_k}^{t_{k+1}}\E\left[\left|\int_{t_k}^{t}\bigl[(g'g)(X(s))-(g'g)(X(t_k))\bigr]\dd B(s)\right|^{2p}\right]\dd t.
\end{align*}
The two terms in the right-hand side of the inequality above are treated separately.

On the one hand, owing to \Cref{assCoeff} the mappings $g'$ and $g''$ are bounded, and for the mappings $g$ and $G$ one has the inequality~\eqref{eq:Gg-growth}. Applying the H\"older inequality and the moment bounds~\eqref{momentExact} from \Cref{propMomEx}, one then obtains the upper bounds
\begin{align*}
\E\left[\Big|\int_{t_k}^{t}\bigl(g'G+\frac12 g''g^2\bigr)(X(s))\dd s\Big|^{2p}\right]
&\le (t-t_k)^{2p-1}\int_{t_k}^{t}\E\left[\Big|\bigl(g'G+\frac12 g''g^2\bigr)(X(s))\Big|^{2p}\right]\dd s\\
&\le C_p\tau^{2p-1}\int_{t_k}^{t}\bigl(\E\left[|X(s)|^{2p}\right]+\E\left[|X(s)|^{4p}\right]\bigr)\dd s\\
&\le C_p(T,x_0)\tau^{2p}.
\end{align*}

On the other hand, applying the BDG inequality~\eqref{eq:BDG} and the H\"older inequality, one has
\begin{align*}
\E[\Big|\int_{t_k}^{t}\bigl[(g'g)(X(s))-(g'g)(X(t_k))\bigr]&\dd B(s)\Big|^{2p}]\le C_p\E[\Bigl(\int_{t_k}^{t}\big|(g'g)(X(s))-(g'g)(X(t_k))\big|^2\dd s\Bigr)^{p}]\\
&\le C_p(t-t_k)^{p-1}\int_{t_k}^{t}\E[\big|(g'g)(X(s))-(g'g)(X(t_k))\big|^{2p}\bigr]\dd s.
\end{align*}
The mapping $g'g$ is not globally Lipschitz continuous on $\R^+$, however it satisfies the inequality~\eqref{eq:boundg'g}. Therefore, applying the regularity property~\eqref{regExNum} from ~\Cref{propregExNum} for the exact solution, the Cauchy--Schwarz inequality, and the moment bounds~\eqref{momentExact} from \Cref{propMomEx}, one obtains the upper bounds
\begin{align*}
\E\left[\Big|\int_{t_k}^{t}\right.&\left.\left[(g'g)(X(s))-(g'g)(X(t_k))\right]\dd B(s)\Big|^{2p}\right]\\
&\le C_p\tau^{p-1}\int_{t_k}^{t}\E\left[\big|X(s)-X(t_k)\big|^{2p}\bigl(1+|X(s)|+|X(t_k)|\bigr)^{2p}\right]\dd s\\
&\le C_p\tau^{p-1}\int_{t_k}^{t}\bigl(\E\left[\big|X(s)-X(t_k)\big|^{4p}\right]\bigr)^{\frac12}\bigl(\E\left[\bigl(1+|X(s)|+|X(t_k)|\bigr)^{4p}\right]\bigr)^{\frac12}\dd s\\
&\le C_p(x_0)\tau^{p-1}\int_{t_k}^{t}(s-t_k)^{p}\dd s\bigl(1+\underset{t\in[0,T]}\sup~\E\left[|X(t)|^{4p}\right]\bigr)^{\frac12}\\
&\le C_p(x_0)\tau^{2p}.
\end{align*}
Combining the upper bounds above, for the error term $\varepsilon_n^{(1,2)}$ one obtains the upper bound
\[
\E\left[|\varepsilon_n^{(1,2)}|^{2p}\right]\le C_p(T,x_0)\tau^{2p}.
\]
This concludes the treatment of the error term $\varepsilon_n^{(1,2)}$.

The proof of Lemma~\ref{lem:aux1} is thus completed.
\end{proof}

Lemma~\ref{lem:aux2} provides upper bounds on the error terms $\varepsilon_n^{(2,j)}$ for $j=1,\ldots,5$.
\begin{lemma}\label{lem:aux2}
Under the setting of \Cref{theoMain}, one has the following error bounds for the error terms defined by Equations~\eqref{err21} to~\eqref{err25}: if the time-step size satisfies the condition $\dt\le \tau_p$, for all $n\in\{1,\ldots,N\}$ one has
\begin{equation}\label{eqErr2}
\E\left[|\varepsilon_n^{(2,j)}|^{2p}\right]\leq C_p(T,x_0)\dt^{2p}, \qquad \forall~j=1,\ldots,5.
\end{equation}
\end{lemma}

The proof of Lemma~\ref{lem:aux2} is divided into three parts to simplify the presentation. The first part deals with the error term $\varepsilon_n^{(2,1)}$ defined by~\eqref{err21}. The second part then deals with the error term $\varepsilon_n^{(2,2)}$ defined by~\eqref{err22}. Finally, the third part deals with the error terms $\varepsilon_n^{(2,3)}$, $\varepsilon_n^{(2,4)}$ and $\varepsilon_n^{(2,5)}$ defined by~\eqref{err23},~\eqref{err24} and~\eqref{err25}.

\begin{proof}[Proof of Lemma~\ref{lem:aux2}: Part 1]
Recall that the error term $\varepsilon_n^{(2,1)}$ is defined by~\eqref{err21}. For all $k\in\{0,\ldots,N-1\}$, on the interval $[t_k,t_{k+1}]$ the auxiliary process $\widetilde{X}$, is solution to the auxiliary stochastic differential equation~\eqref{auxSDE}, thus for all $t\in(t_k,t_{k+1})$, one has
\begin{align*}
\widetilde{X}(t)-\widetilde{X}(t_k)&=\int_{t_k}^{t}\widetilde{X}(s)F(X_k)\dd s\\
&+\int_{t_k}^{t}\widetilde{X}(s)f(X_k)\dd B(s)\\
&+\int_{t_k}^{t}\widetilde{X}(s)\bigl(f'g\bigr)(X_k)\bigl(B(s)-B(t_k)\bigr)\dd B(s)\\
&+\int_{t_k}^{t}\widetilde{X}(s)\bigl(ff'g\bigr)(X_k)\bigl(B(s)-B(t_k)\bigr)\dd s\\
&+\frac12\int_{t_k}^{t}\widetilde{X}(s)\bigl(f'g\bigr)^2(X_k)\bigl(B(s)-B(t_k)\bigr)^2\dd s.
\end{align*}
Integrating on the interval $[t_k,t_{k+1}]$ and summing for $k=0,\ldots,n-1$, the error term $\varepsilon_n^{(2,1)}$ is then decomposed as
\begin{align*}
\varepsilon_n^{(2,1)}
&=\sum_{k=0}^{n-1}\int_{t_k}^{t_{k+1}}F(X_k)\Bigl(\widetilde{X}(t)-\widetilde{X}(t_k)\Bigr)\dd t\\
&=\varepsilon_n^{(2,1,1)}+\varepsilon_n^{(2,1,2)}+\varepsilon_n^{(2,1,3)}+\varepsilon_n^{(2,1,4)}+\varepsilon_n^{(2,1,5)},
\end{align*}
where the auxiliary error terms are defined as
\begin{align*}
\varepsilon_n^{(2,1,1)}&=\sum_{k=0}^{n-1}\int_{t_k}^{t_{k+1}}\int_{t_k}^{t}\widetilde{X}(s)F^2(X_k)\dd s\dd t\\
\varepsilon_n^{(2,1,2)}&=\sum_{k=0}^{n-1}\int_{t_k}^{t_{k+1}}\int_{t_k}^{t}\widetilde{X}(s)\bigl(Ff\bigr)(X_k)\dd B(s)\dd t\\
\varepsilon_n^{(2,1,3)}&=\sum_{k=0}^{n-1}\int_{t_k}^{t_{k+1}}\int_{t_k}^{t}\widetilde{X}(s)\bigl(Ff'g\bigr)(X_k)\bigl(B(s)-B(t_k)\bigr)\dd B(s)\dd t\\
\varepsilon_n^{(2,1,4)}&=\sum_{k=0}^{n-1}\int_{t_k}^{t_{k+1}}\int_{t_k}^{t}\widetilde{X}(s)\bigl(Fff'g\bigr)(X_k)\bigl(B(s)-B(t_k)\bigr)\dd s\dd t\\
\varepsilon_n^{(2,1,5)}&=\frac12\sum_{k=0}^{n-1}\int_{t_k}^{t_{k+1}}\int_{t_k}^{t}\widetilde{X}(s)\bigl(F(f'g)^2\bigr)(X_k)\bigl(B(s)-B(t_k)\bigr)^2\dd s\dd t.
\end{align*}
The remainder of the proof is devoted to proving upper bounds for the five auxiliary error terms defined above.

{\bf Treatment of the error term $\varepsilon_{n}^{(2,1,1)}$.}

The mapping $F$ is bounded (see~\eqref{eq:boundF}). Applying the Minkowski inequality and the moment bounds~\eqref{momentNum} from \Cref{propMomNum} on the auxiliary process $\widetilde{X}$, one obtains the upper bounds
\begin{align*}
\left(\E\left[\left|\sum_{k=0}^{n-1}\int_{t_k}^{t_{k+1}}\int_{t_k}^{t}\widetilde{X}(s)F^2(X_k)\dd s\dd t\right|^{2p}\right]\right)^{\frac{1}{2p}}
&\le\sum_{k=0}^{n-1}\int_{t_k}^{t_{k+1}}\int_{t_k}^{t}\left(\E\left[\left|\widetilde{X}(s)F^2(X_k)\right|^{2p}\right]\right)^{\frac{1}{2p}}\dd s\dd t\\
&\le C\sum_{k=0}^{n-1}\int_{t_k}^{t_{k+1}}\int_{t_k}^{t}\dd s\dd t\underset{s\in[0,T]}\sup~\left(\E\left[|\widetilde{X}(s)|^{2p}\right]\right)^{\frac{1}{2p}}\\
&\le C_p(T,x_0)\tau.
\end{align*}

{\bf Treatment of the error term $\varepsilon_{n}^{(2,1,2)}$.}

Applying the stochastic Fubini theorem, one has
\[
\int_{t_k}^{t_{k+1}}\int_{t_k}^{t}\widetilde{X}(s)\bigl(Ff\bigr)(X_k)\dd B(s)\dd t=\int_{t_k}^{t_{k+1}}(t_{k+1}-s)\widetilde{X}(s)\bigl(Ff\bigr)(X_k)\dd B(s).
\]
The mappings $F$ and $f$ are bounded (see~\eqref{eq:boundF} and~\eqref{eq:boundf}). Applying the discrete-time version~\eqref{eq:BDGdiscrete} of the BDG inequality and the moment bounds~\eqref{momentNum} from \Cref{propMomNum} on the auxiliary process $\widetilde{X}$, one obtains the upper bounds
\begin{align*}
\E\left[\left|\sum_{k=0}^{n-1}\int_{t_k}^{t_{k+1}}\right.\right.&\left.\left.(t_{k+1}-s)\widetilde{X}(s)\bigl(Ff\bigr)(X_k)\dd B(s)\right|^{2p}\right]\\
&\le C_p(T)\sum_{k=0}^{n-1}\int_{t_k}^{t_{k+1}}(t_{k+1}-s)^{2p}\E\left[\left|\widetilde{X}(s)\bigl(Ff\bigr)(X_k)\right|^{2p}\right]\dd s\\
&\le C_p(T)\tau^{2p}\underset{s\in[0,T]}\sup~\E\left[|\widetilde{X}(s)|^{2p}\right]\\
&\le C_p(T,x_0)\tau^{2p}.
\end{align*}

{\bf Treatment of the error term $\varepsilon_{n}^{(2,1,3)}$.}

Applying the stochastic Fubini theorem, one has
\begin{align*}
\int_{t_k}^{t_{k+1}}\int_{t_k}^{t}&\widetilde{X}(s)\bigl(Ff'g\bigr)(X_k)\bigl(B(s)-B(t_k)\bigr)\dd B(s)\dd t\\
&=\int_{t_k}^{t_{k+1}}(t_{k+1}-s)\widetilde{X}(s)\bigl(Ff'g\bigr)(X_k)\bigl(B(s)-B(t_k)\bigr)\dd B(s).
\end{align*}
The mappings $F$ and $f'g$ are bounded (see~\eqref{eq:boundF} and~\eqref{eq:boundf'g}). Applying the discrete-time version~\eqref{eq:BDGdiscrete} of the BDG inequality, the Cauchy--Schwarz inequality, and the moment bounds~\eqref{momentNum} from \Cref{propMomNum} on the auxiliary process $\widetilde{X}$, one obtains the upper bounds
\begin{align*}
\E\left[\left|\sum_{k=0}^{n-1}\int_{t_k}^{t_{k+1}}\right.\right.&\left.\left.
(t_{k+1}-s)\widetilde{X}(s)\bigl(Ff'g\bigr)(X_k)\bigl(B(s)-B(t_k)\bigr)\dd B(s)\right|^{2p}\right]\\
&\le C_p(T)\sum_{k=0}^{n-1}\int_{t_k}^{t_{k+1}}(t_{k+1}-s)^{2p}\E\left[|\widetilde{X}(s)\bigl(Ff'g\bigr)(X_k)\bigl(B(s)-B(t_k)|^{2p}\right]\dd s\\
&\le C_p(T)\tau^{2p}\sum_{k=0}^{n-1}\int_{t_k}^{t_{k+1}}\bigl(\E\bigl[|\widetilde{X}(s)|^{4p}\bigr]\bigr)^{\frac12}\bigl(\E\left[ |B(s)-B(t_k)|^{4p}\right]\bigr)^{\frac12}\dd s\\
&\le C_p(T)\tau^{2p}\sum_{k=0}^{n-1}\int_{t_k}^{t_{k+1}}(s-t_k)^{p}\dd s\underset{s\in[0,T]}\sup~\E\left[|\widetilde{X}(s)|^{2p}\right]\\
&\le C_p(T,x_0)\tau^{3p}.
\end{align*}

{\bf Treatment of the error term $\varepsilon_{n}^{(2,1,4)}$.}

The mappings $F$, $f$ and $f'g$ are bounded (see~\eqref{eq:boundF},~\eqref{eq:boundf} and~\eqref{eq:boundf'g}). Applying the Minkowski and Cauchy--Schwarz inequalities, and the moment bounds~\eqref{momentNum} from \Cref{propMomNum} on the auxiliary process $\widetilde{X}$, one obtains the upper bounds
\begin{align*}
\left(\E\left[\left|\right.\right.\right.&\left.\left.\left.
\sum_{k=0}^{n-1}\int_{t_k}^{t_{k+1}}\int_{t_k}^{t}\widetilde{X}(s)\bigl(Fff'g\bigr)(X_k)\bigl(B(s)-B(t_k)\bigr)\dd s\dd t\right|^{2p}\right]\right)^{\frac{1}{2p}}\\
&\le \sum_{k=0}^{n-1}\int_{t_k}^{t_{k+1}}\int_{t_k}^{t}\left(\E\left[\big|\widetilde{X}(s)\bigl(Fff'g\bigr)(X_k)\bigl(B(s)-B(t_k)\bigr)\big|^{2p}\right]\right)^{\frac{1}{2p}}\dd s \dd t\\
&\le C_p(T)\sum_{k=0}^{n-1}\int_{t_k}^{t_{k+1}}\int_{t_k}^{t}\bigl(\E\bigl[|\widetilde{X}(s)|^{4p}\bigr]\bigr)^{\frac{1}{4p}}\bigl(\E\left[ |B(s)-B(t_k)|^{4p}\right]\bigr)^{\frac{1}{4p}}\dd s \dd t\\
&\le C_p\sum_{k=0}^{n-1}\int_{t_k}^{t_{k+1}}\int_{t_k}^{t} (s-t_k)^{\frac12} \dd s \dd t \underset{s\in[0,T]}\sup~\left(\E\left[|\widetilde{X}(s)|^{4p}\right]\right)^{\frac{1}{4p}}\\
&\le C_p(T,x_0)\tau^{\frac32}.
\end{align*}

{\bf Treatment of the error term $\varepsilon_{n}^{(2,1,5)}$.}

The mappings $F$ and $f'g$ are bounded (see~\eqref{eq:boundF}  and~\eqref{eq:boundf'g}). Applying the Minkowski and Cauchy--Schwarz inequalities, and the moment bounds~\eqref{momentNum} from \Cref{propMomNum} on the auxiliary process $\widetilde{X}$, one obtains the upper bounds
\begin{align*}
&\bigg( \E \bigg[ \Big|
\sum_{k=0}^{n-1}\int_{t_k}^{t_{k+1}}\int_{t_k}^{t}\widetilde{X}(s)\bigl(F(f'g)^2\bigr)(X_k)\bigl(B(s)-B(t_k)\bigr)^2\dd s\dd t\Big|^{2p} \bigg] \bigg)^{\frac{1}{2p}}\\
&\quad\le \sum_{k=0}^{n-1}\int_{t_k}^{t_{k+1}}\int_{t_k}^{t}
\left( \E \left[ \big|\widetilde{X}(s)\bigl(F(f'g)^2\bigr)(X_k)\bigl(B(s)-B(t_k)\bigr)^2\big|^{2p} \right] \right)^{\frac{1}{2p}}\dd s\dd t\\
&\quad \le \sum_{k=0}^{n-1}\int_{t_k}^{t_{k+1}}\int_{t_k}^{t}\bigl(\E\bigl[|\widetilde{X}(s)|^{4p}\bigr]\bigr)^{\frac{1}{4p}}\bigl(\E\left[ |B(s)-B(t_k)|^{8p}\right]\bigr)^{\frac{1}{4p}}\dd s \dd t\\
&\quad \le C_p\sum_{k=0}^{n-1}\int_{t_k}^{t_{k+1}}\int_{t_k}^{t} (s-t_k) \dd s \dd t \underset{s\in[0,T]}\sup~\left(\E\left[|\widetilde{X}(s)|^{4p}\right]\right)^{\frac{1}{4p}}\\
&\quad\le C_p(T,x_0)\tau^{2}.
\end{align*}

{\bf Conclusion.}

Gathering the upper bounds obtained above for the auxiliary error terms $\varepsilon_n^{(2,1,j)}$ for $j=1,\ldots,5$, one obtains the following upper bound for the error term $\varepsilon_n^{(2,1)}$: one has
\[
\E\left[|\varepsilon_n^{(2,1)}|^{2p}\right]\le C_p(T,x_0)\tau^{2p}.
\]
This concludes the first part of the proof of Lemma~\ref{lem:aux2}.
\end{proof}

\begin{proof}[Proof of Lemma~\ref{lem:aux2}: Part 2]
Recall that the error term $\varepsilon_n^{(2,2)}$ is defined by~\eqref{err22}. In addition, recall that owing to~\eqref{eq:facto} one has the factorization property $g(x)=xf(x)$ for all $x\in\R^+$. For all $k\in\{0,\ldots,N-1\}$, on the interval $[t_k,t_{k+1}]$ the auxiliary process $\widetilde{X}$, is solution to the auxiliary stochastic differential equation~\eqref{auxSDE}, thus for all $t\in(t_k,t_{k+1})$, one has
\begin{align*}
\widetilde{X}(t)-\widetilde{X}(t_k)-g(X_k)\bigl[B(t)-B(t_k)\bigr]
&=\int_{t_k}^{t}\widetilde{X}(s)F(X_k)\dd s\\
&+\int_{t_k}^{t}\bigl[\widetilde{X}(s)-\widetilde{X}(t_k)\bigr]f(X_k)\dd B(s)\\
&+\int_{t_k}^{t}\widetilde{X}(s)\bigl(f'g\bigr)(X_k)\bigl(B(s)-B(t_k)\bigr)\dd B(s)\\
&+\int_{t_k}^{t}\widetilde{X}(s)\bigl(ff'g\bigr)(X_k)\bigl(B(s)-B(t_k)\bigr)\dd s\\
&+\frac12\int_{t_k}^{t}\widetilde{X}(s)\bigl(f'g\bigr)^2(X_k)\bigl(B(s)-B(t_k)\bigr)^2\dd s.
\end{align*}
Integrating on the interval $[t_k,t_{k+1}]$ and summing for $k=0,\ldots,n-1$, the error term $\varepsilon_n^{(2,2)}$ is then decomposed as
\begin{align*}
\varepsilon_n^{(2,2)}
&=\sum_{k=0}^{n-1}\int_{t_k}^{t_{k+1}}f(X_k)\Bigl(\widetilde{X}(t)-\widetilde{X}(t_k)-g(X_k)\bigl[B(t)-B(t_k)\bigr]\Bigr)\dd B(t)\\
&=\varepsilon_n^{(2,2,1)}+\varepsilon_n^{(2,2,2)}+\varepsilon_n^{(2,2,3)}+\varepsilon_n^{(2,2,4)}+\varepsilon_n^{(2,2,5)},
\end{align*}
where the auxiliary error terms are defined as
\begin{align*}
\varepsilon_n^{(2,2,1)}&=\sum_{k=0}^{n-1}\int_{t_k}^{t_{k+1}}\int_{t_k}^{t}f(X_k)\widetilde{X}(s)F(X_k)\dd s\dd B(t)\\
\varepsilon_n^{(2,2,2)}&=\sum_{k=0}^{n-1}\int_{t_k}^{t_{k+1}}\int_{t_k}^{t}f^2(X_k)\bigl[\widetilde{X}(s)-\widetilde{X}(t_k)\bigr]\dd B(s)\dd B(t)\\
\varepsilon_n^{(2,2,3)}&=\sum_{k=0}^{n-1}\int_{t_k}^{t_{k+1}}\int_{t_k}^{t}\widetilde{X}(s)\bigl(ff'g\bigr)(X_k)\bigl(B(s)-B(t_k)\bigr)\dd B(s)\dd B(t)\\
\varepsilon_n^{(2,2,4)}&=\sum_{k=0}^{n-1}\int_{t_k}^{t_{k+1}}\int_{t_k}^{t}\widetilde{X}(s)\bigl(f^2f'g\bigr)(X_k)\bigl(B(s)-B(t_k)\bigr)\dd s\dd B(t)\\
\varepsilon_n^{(2,2,5)}&=\frac12\sum_{k=0}^{n-1}\int_{t_k}^{t_{k+1}}\int_{t_k}^{t}\widetilde{X}(s)\bigl(f(f'g)^2\bigr)(X_k)\bigl(B(s)-B(t_k)\bigr)^2\dd s \dd B(t).
\end{align*}
The remainder of the proof is devoted to proving upper bounds for the five auxiliary error terms defined above.

{\bf Treatment of the error term $\varepsilon_{n}^{(2,2,1)}$.}

The mappings $F$ and $f$ are bounded (see~\eqref{eq:boundF} and~\eqref{eq:boundf}). Applying the discrete-time version~\eqref{eq:BDGdiscrete} of the BDG inequality, the H\"older inequality, and the moment bounds~\eqref{momentNum} from \Cref{propMomNum} on the auxiliary process $\widetilde{X}$, one obtains the upper bounds
\begin{align*}
\E\Biggl[\Bigl|&\sum_{k=0}^{n-1}\int_{t_k}^{t_{k+1}}\int_{t_k}^{t}f(X_k)\widetilde{X}(s)F(X_k)\dd s\dd B(t)\Bigr|^{2p}\Biggr]\\
&\le C_p(T)\sum_{k=0}^{n-1}\int_{t_k}^{t_{k+1}}\E\Biggl[\Bigl|\int_{t_k}^{t}f(X_k)\widetilde{X}(s)F(X_k)\dd s\Bigr|^{2p}\Biggr]\dd t\\
&\le C_p(T)\sum_{k=0}^{n-1}\int_{t_k}^{t_{k+1}}(t-t_k)^{2p-1}\int_{t_k}^{t}\E\left[\big|f(X_k)\widetilde{X}(s)F(X_k)\big|^{2p}\right]\dd s\dd t\\
&\le C_p(T)\tau^{2p}\underset{s\in[0,T]}\sup~\E\left[|\widetilde{X}(s)|^{2p}\right]\\
&\le C_p(T,x_0)\tau^{2p}.
\end{align*}

{\bf Treatment of the error term $\varepsilon_{n}^{(2,2,2)}$.}

The mapping $f$ is bounded (see~\eqref{eq:boundf}). Applying the discrete-time version~\eqref{eq:BDGdiscrete} of the BDG inequality, the first part of the BDG inequality~\eqref{eq:BDG}, and the H\"older inequality, one obtains the upper bounds
\begin{align*}
\E\Biggl[\Bigl|&\sum_{k=0}^{n-1}\int_{t_k}^{t_{k+1}}\int_{t_k}^{t}f^2(X_k)\bigl[\widetilde{X}(s)-\widetilde{X}(t_k)\bigr]\dd B(s)\dd B(t)\Bigr|^{2p}\Biggr]\\
&\le C_p(T)\sum_{k=0}^{n-1}\int_{t_k}^{t_{k+1}}\E\Biggl[\Bigl|\int_{t_k}^{t}f^2(X_k)\bigl[\widetilde{X}(s)-\widetilde{X}(t_k)\bigr]\dd B(s)\Bigr|^{2p}\Biggr]\dd t\\
&\le C_p(T)\sum_{k=0}^{n-1}\int_{t_k}^{t_{k+1}}\E\Biggl[\Bigl(\int_{t_k}^{t}\big|f^2(X_k)\bigl[\widetilde{X}(s)-\widetilde{X}(t_k)\big|^2]\dd s\Bigr)^{p}\Biggr]\dd t\\
&\le C_p(T)\sum_{k=0}^{n-1}\int_{t_k}^{t_{k+1}}(t-t_k)^{p-1}\int_{t_k}^{t}\E\left[\big|\widetilde{X}(s)-\widetilde{X}(t_k)\big|^{2p}\right]\dd s\dd t.
\end{align*}
Applying the increment bounds~\eqref{regExNum} from Proposition~\ref{propregExNum} on the auxiliary process $\widetilde{X}$, one then obtains
\begin{align*}
\E\Biggl[\Bigl|&\sum_{k=0}^{n-1}\int_{t_k}^{t_{k+1}}\int_{t_k}^{t}f(X_k)^2\bigl[\widetilde{X}(s)-\widetilde{X}(t_k)\bigr]\dd B(s)\dd B(t)\Bigr|^{2p}\Biggr]\\
&\le C_p(T,x_0)\sum_{k=0}^{n-1}\int_{t_k}^{t_{k+1}}(t-t_k)^{p-1}\int_{t_k}^{t}(s-t_k)^{p}\dd s\dd t\\
&\le C_p(T,x_0)\tau^{2p}.
\end{align*}

{\bf Treatment of the error term $\varepsilon_{n}^{(2,2,3)}$.}

The mappings $f$ and $f'g$ are bounded (see~\eqref{eq:boundf} and~\eqref{eq:boundf'g}). Applying the discrete-time version~\eqref{eq:BDGdiscrete} of the BDG inequality, the first part of the BDG inequality~\eqref{eq:BDG}, and the H\"older and Cauchy--Schwarz inequalities, one obtains the upper bounds
\begin{align*}
\E\Biggl[\Bigl|&\sum_{k=0}^{n-1}\int_{t_k}^{t_{k+1}}\int_{t_k}^{t}\widetilde{X}(s)\bigl(ff'g\bigr)(X_k)\bigl(B(s)-B(t_k)\bigr)\dd B(s)\dd B(t)\Bigr|^{2p}\Biggr]\\
&\le C_p(T)\sum_{k=0}^{n-1}\int_{t_k}^{t_{k+1}}\E\Biggl[\Bigl|\int_{t_k}^{t}\widetilde{X}(s)\bigl(ff'g\bigr)(X_k)\bigl(B(s)-B(t_k)\bigr)\dd B(s)\Bigr|^{2p}\Biggr]\dd t\\
&\le C_p(T)\sum_{k=0}^{n-1}\int_{t_k}^{t_{k+1}}\Bigl(\int_{t_k}^{t} \E\left[\big|\widetilde{X}(s)\bigl(ff'g\bigr)(X_k)\bigl(B(s)-B(t_k)\bigr)\big|^2\right]\dd s\Bigr)^p\dd t\\
&\le C_p(T)\sum_{k=0}^{n-1}\int_{t_k}^{t_{k+1}}(t-t_k)^{p-1}\int_{t_k}^{t}\E\left[\big|\widetilde{X}(s)\bigl(B(s)-B(t_k)\bigr)\big|^{2p}\right]\dd s \dd t\\
&\le C_p(T)\sum_{k=0}^{n-1}\int_{t_k}^{t_{k+1}}(t-t_k)^{p-1}\int_{t_k}^{t}\left(\E\left[|\widetilde{X}(s)|^{4p}\right]\right)^{\frac12}
\left(\E\left[|B(s)-B(t_k)|^{4p}\right]\right)^{\frac12}\dd s \dd t.
\end{align*}
Applying the moment bounds~\eqref{momentNum} from \Cref{propMomNum} on the auxiliary process $\widetilde{X}$, one obtains the upper bounds
\begin{align*}
\E\Biggl[\Bigl|&\sum_{k=0}^{n-1}\int_{t_k}^{t_{k+1}}\int_{t_k}^{t}f(X_k)\widetilde{X}(s)\bigl(f'g\bigr)(X_k)\bigl(B(s)-B(t_k)\bigr)\dd B(s)\dd B(t)\Bigr|^{2p}\Biggr]\\
&\le C_p(T,x_0)\sum_{k=0}^{n-1}\int_{t_k}^{t_{k+1}}(t-t_k)^{p-1}\int_{t_k}^{t}(s-t_k)^p \dd s \dd t\\
&\le C_p(T,x_0)\tau^{2p}.
\end{align*}

{\bf Treatment of the error term $\varepsilon_{n}^{(2,2,4)}$.}

The mappings $f$ and $f'g$ are bounded (see~\eqref{eq:boundf} and~\eqref{eq:boundf'g}). Applying the discrete-time version~\eqref{eq:BDGdiscrete} of the BDG inequality and the H\"older inequality, one obtains the upper bounds
\begin{align*}
\E\Biggl[\Bigl|&\sum_{k=0}^{n-1}\int_{t_k}^{t_{k+1}}\int_{t_k}^{t}\widetilde{X}(s)\bigl(f^2f'g\bigr)(X_k)\bigl(B(s)-B(t_k)\bigr)\dd s\dd B(t)\Bigr|^{2p}\Biggr]\\
&\le C_p(T)\sum_{k=0}^{n-1}\int_{t_k}^{t_{k+1}}\E\Biggl[\Bigl|\int_{t_k}^{t}\widetilde{X}(s)\bigl(f^2f'g\bigr)(X_k)\bigl(B(s)-B(t_k)\bigr)\dd s\Bigr|^{2p}\Biggr]\dd t\\
&\le C_p(T)\sum_{k=0}^{n-1}\int_{t_k}^{t_{k+1}}(t-t_k)^{2p-1}\int_{t_k}^{t}\E\left[\big|\widetilde{X}(s)\bigl(f^2f'g\bigr)(X_k)\bigl(B(s)-B(t_k)\bigr)\big|^{2p}\right]\dd s \dd t\\
&\le C_p(T)\sum_{k=0}^{n-1}\int_{t_k}^{t_{k+1}}(t-t_k)^{2p-1}\int_{t_k}^{t}\bigl(\E[\big|\widetilde{X}(s)\big|^{4p}]\bigr)^{\frac12}\bigl(\E[\big|B(s)-B(t_k)\big|^{4p}]\bigr)^{\frac12}\dd s\dd t.
\end{align*}
Applying the moment bounds~\eqref{momentNum} from \Cref{propMomNum} on the auxiliary process $\widetilde{X}$, one obtains the upper bounds
\begin{align*}
\E\Biggl[\Bigl|&\sum_{k=0}^{n-1}\int_{t_k}^{t_{k+1}}\int_{t_k}^{t}f(X_k)\widetilde{X}(s)f(X_k)\bigl(f'g\bigr)(X_k)\bigl(B(s)-B(t_k)\bigr)\dd s\dd B(t)\Bigr|^{2p}\Biggr]\\
&\le C_p(T,x_0)\sum_{k=0}^{n-1}\int_{t_k}^{t_{k+1}}(t-t_k)^{2p-1}\int_{t_k}^{t}(s-t_k)^p \dd s \dd t\\
&\le C_p(T,x_0)\tau^{3p}.
\end{align*}

{\bf Treatment of the error term $\varepsilon_{n}^{(2,2,5)}$.}

The mappings $f$ and $f'g$ are bounded (see~\eqref{eq:boundf} and~\eqref{eq:boundf'g}). Applying the discrete-time version~\eqref{eq:BDGdiscrete} of the BDG inequality and the H\"older inequality, one obtains the upper bounds
\begin{align*}
\E[\big|&\sum_{k=0}^{n-1}\int_{t_k}^{t_{k+1}}\int_{t_k}^{t}\widetilde{X}(s)\bigl(f(f'g)^2\bigr)(X_k)\bigl(B(s)-B(t_k)\bigr)^2\dd s \dd B(t)\big|^{2p}]\\
&\le C_p(T)\sum_{k=0}^{n-1}\int_{t_k}^{t_{k+1}}\E[\big|\int_{t_k}^{t}\widetilde{X}(s)\bigl(f(f'g)^2\bigr)\bigl(B(s)-B(t_k)\bigr)^2\dd s\big|^{2p}]\dd t\\
&\le C_p(T)\sum_{k=0}^{n-1}\int_{t_k}^{t_{k+1}}(t-t_k)^{2p-1}\int_{t_k}^{t}\E[\big|\widetilde{X}(s)\bigl(f(f'g)^2\bigr)(X_k)\bigl(B(s)-B(t_k)\bigr)^2\big|^{2p}]\dd s\dd t\\
&\le C_p(T)\sum_{k=0}^{n-1}\int_{t_k}^{t_{k+1}}(t-t_k)^{2p-1}\int_{t_k}^{t}\bigl(\E[\big|\widetilde{X}(s)\big|^{4p}]\bigr)^{\frac12}\bigl(\E[\big|B(s)-B(t_k)\big|^{8p}]\bigr)^{\frac12}\dd s\dd t.
\end{align*}
Applying the moment bounds~\eqref{momentNum} from \Cref{propMomNum} on the auxiliary process $\widetilde{X}$, one obtains the upper bounds
\begin{align*}
\E[\big|&\sum_{k=0}^{n-1}\int_{t_k}^{t_{k+1}}\int_{t_k}^{t}f(X_k)\widetilde{X}(s)\bigl(f'g\bigr)^2(X_k)\bigl(B(s)-B(t_k)\bigr)^2\dd s \dd B(t)\big|^{2p}]\\
&\le C_p(T,x_0)\sum_{k=0}^{n-1}\int_{t_k}^{t_{k+1}}(t-t_k)^{2p-1}\int_{t_k}^{t}(s-t_k)^{2p} \dd s \dd t\\
&\le C_p(T,x_0)\tau^{4p}.
\end{align*}

{\bf Conclusion.}

Gathering the upper bounds obtained above for the auxiliary error terms $\varepsilon_n^{(2,2,j)}$ for $j=1,\ldots,5$, one obtains the following upper bound for the error term $\varepsilon_n^{(2,2)}$: one has
\[
\E\left[|\varepsilon_n^{(2,2)}|^{2p}\right]\le C_p(T,x_0)\tau^{2p}.
\]
This concludes the second part of the proof of Lemma~\ref{lem:aux2}.
\end{proof}

\begin{proof}[Proof of Lemma~\ref{lem:aux2}: Part 3]

This final part of the proof is devoted to the treatment of the error terms $\varepsilon_n^{(2,3)}$, $\varepsilon_n^{(2,4)}$ and $\varepsilon_n^{(2,5)}$ defined by~\eqref{err23},~\eqref{err24} and~\eqref{err25}.

{\bf Treatment of the error term $\varepsilon_{n}^{(2,3)}$.}

The mapping $f'g$ is bounded (see~\eqref{eq:boundf'g}). Applying the discrete-time version~\eqref{eq:BDGdiscrete} of the BDG inequality, the Cauchy--Schwarz inequality, and the increment bounds~\eqref{regExNum} from Proposition~\ref{propregExNum} on the auxiliary process $\widetilde{X}$, one obtains the upper bounds
\begin{align*}
\E\left[|\varepsilon_n^{(2,3)}|^{2p}\right]&\le C_p(T)\sum_{k=0}^{n-1}\int_{t_k}^{t_{k+1}}\E\left[\big|\bigl(f'g\bigr)(X_k)\Bigl(\widetilde{X}(t)-\widetilde{X}(t_k)\Bigr)\bigl[B(t)-B(t_k)\bigr]\big|^{2p}\right]\dd t\\
&\le C_p(T)\sum_{k=0}^{n-1}\int_{t_k}^{t_{k+1}}
\left(\E\left[|\widetilde{X}(t)-\widetilde{X}(t_k)|^{4p}\right]\right)^{\frac12}
\left(\E\left[\big|B(t)-B(t_k)\big|^{4p}\right]\right)^{\frac12}\dd t\\
&\le C_p(T,x_0)\sum_{k=0}^{n-1}\int_{t_k}^{t_{k+1}}(t-t_k)^{2p}\dd t\\
&\le C_p(T,x_0)\tau^{2p}.
\end{align*}

{\bf Treatment of the error term $\varepsilon_{n}^{(2,4)}$.}

Recalling that one has $\widetilde{X}(t_k)=X_k$ for all $k\in\{0,\ldots,N\}$, the error term $\varepsilon_n^{(2,4)}$ is decomposed as
\begin{align*}
\varepsilon_n^{(2,4)}&=\sum_{k=0}^{n-1}\int_{t_k}^{t_{k+1}}\bigl(ff'g\bigr)(X_k)\widetilde{X}(t)\bigl[B(t)-B(t_k)\bigr]\dd t\\
&=\sum_{k=0}^{n-1}\int_{t_k}^{t_{k+1}}\bigl(ff'g\bigr)(X_k)\bigl[\widetilde{X}(t)-\widetilde{X}(t_k)\bigr]\bigl[B(t)-B(t_k)\bigr]\dd t\\
&+\sum_{k=0}^{n-1}\int_{t_k}^{t_{k+1}}\bigl(ff'g\bigr)(X_k)X_k\bigl[B(t)-B(t_k)\bigr]\dd t.
\end{align*}

The mappings $f$ and $f'g$ are bounded (see~\eqref{eq:boundf} and~\eqref{eq:boundf'g}).

On the one hand, applying the Minkowski and Cauchy--Schwarz inequalities, and the increment bounds~\eqref{regExNum} from Proposition~\ref{propregExNum} on the auxiliary process $\widetilde{X}$ one obtains the upper bounds
\begin{align*}
\Biggl(\E\Bigl[\Bigl|&\sum_{k=0}^{n-1}\int_{t_k}^{t_{k+1}}\bigl(ff'g\bigr)(X_k)\bigl[\widetilde{X}(t)-\widetilde{X}(t_k)\bigr]\bigl[B(t)-B(t_k)\bigr]\dd t\Bigr|^{2p}\Bigr]\Biggr)^{\frac{1}{2p}}\\
&\le \sum_{k=0}^{n-1}\int_{t_k}^{t_{k+1}}\Bigl(\E\left[\big| \bigl(ff'g\bigr)(X_k)\bigl[\widetilde{X}(t)-\widetilde{X}(t_k)\bigr]\bigl[B(t)-B(t_k)\bigr]\big|^{2p}\right]\Bigr)^{\frac{1}{2p}}\dd t\\
&\le C\sum_{k=0}^{n-1}\int_{t_k}^{t_{k+1}}\bigl(\E[\big|\widetilde{X}(t)-\widetilde{X}(t_k)\big|^{4p}]\bigr)^{\frac{1}{4p}}\bigl(\E[\big|B(t)-B(t_k)\big|^{4p}]\bigr)^{\frac{1}{4p}}\dd t\\
&\le C_p(T,x_0)\sum_{k=0}^{n-1}\int_{t_k}^{t_{k+1}}(t-t_k)^{2p}\dd t\\
&\le C_p(T,x_0)\tau.
\end{align*}
On the other hand, applying the stochastic Fubini theorem one has
\begin{align*}
\sum_{k=0}^{n-1}\int_{t_k}^{t_{k+1}}\bigl(ff'g\bigr)(X_k)X_k\bigl[B(t)-B(t_k)\bigr]\dd t&=\sum_{k=0}^{n-1}\int_{t_k}^{t_{k+1}}\int_{t_k}^{t}\bigl(ff'g\bigr)(X_k)X_k\dd B(s)\dd t\\
&=\sum_{k=0}^{n-1}\int_{t_k}^{t_{k+1}}(t_{k+1}-s)\bigl(ff'g\bigr)(X_k)X_k\dd B(s).
\end{align*}
Applying the discrete-time version~\eqref{eq:BDGdiscrete} of the BDG inequality and the moment bounds~\eqref{momentNum} from \Cref{propMomNum} on the numerical solution, one obtains the upper bounds
\begin{align*}
\E\Bigl[\big|\sum_{k=0}^{n-1}&\int_{t_k}^{t_{k+1}}(t_{k+1}-s)\bigl(ff'g\bigr)(X_k)X_k\dd B(s)\big|^{2p}\Bigr]\\
&\le C_p(T)\sum_{k=0}^{n-1}\int_{t_k}^{t_{k+1}}(t_{k+1}-s)^{2p}\E\bigl[\big|\bigl(ff'g\bigr)(X_k)X_k|^{2p}\bigr]\dd s\\
&\le C_p(T)\tau^{2p}\underset{0\le n\le N}\sup~\E\left[|X_k|^{2p}\right]\\
&\le C_p(T,x_0)\tau^{2p}.
\end{align*}
Combining the upper bounds above, for the auxiliary error term $\varepsilon_n^{(2,4)}$ one has
\[
\E\left[|\varepsilon_n^{(2,4)}|^{2p}\right]\le C_p(T,x_0)\tau^{2p}.
\]

{\bf Treatment of the error term $\varepsilon_{n}^{(2,5)}$.}

The mapping $f'g$ is bounded. Applying the Minkowski and Cauchy--Schwarz inequalities, and the moment bounds~\eqref{momentNum} from \Cref{propMomNum} on the auxiliary process $\widetilde{X}$, one obtains the upper bounds
\begin{align*}
\left(\E\left[|\varepsilon_n^{(2,5)}|^{2p}\right]\right)^{\frac{1}{2p}}&\le \frac12 \sum_{k=0}^{n-1}\int_{t_k}^{t_{k+1}}
\left(\E\left[\big|\bigl(f'g\bigr)^2(X_k)\widetilde{X}(t)\bigl[B(t)-B(t_k)\bigr]^2\big|^{2p}\right]\right)^{\frac{1}{2p}}\dd t\\
&\le C_p\sum_{k=0}^{n-1}\int_{t_k}^{t_{k+1}}
\left(\E\left[|\widetilde{X}(t)|^{4p}\right]\right)^{\frac{1}{4p}}
\left(\E\left[\big|B(t)-B(t_k)\big|^{8p}\right]\right)^{\frac{1}{4p}}\dd t\\
&\le C_p(T,x_0)\sum_{k=0}^{n-1}\int_{t_k}^{t_{k+1}}(t-t_k)\dd t\\
&\le C_p(T,x_0)\tau.
\end{align*}

This concludes the third part of the proof of Lemma~\ref{lem:aux2}.
\end{proof}

\subsection{Proof of strong error estimates}\label{secOrder}

This section is dedicated to the presentation of the final arguments of the proof of \Cref{theoMain}, using the ingredients obtained in Sections~\ref{secDec} and~\eqref{secAuxErr}. Recall the decomposition~\eqref{eqDecErr} of the error $e_n=X_n-X(t_n)$. Upper bounds for the error terms $\varepsilon_n^{(1,j)}$ ($j=1,2$) and $\varepsilon_n^{(2,j)}$ ($j=1,\ldots,5$) are given by Lemmas~\ref{lem:aux1} and~\ref{lem:aux2}. It remains to deal now with the error terms $\varepsilon_n^{(3,j)}$ ($j=1,2,3$). On the one hand, the error terms $\varepsilon_n^{(3,1)}$ and $\varepsilon_n^{(3,2)}$ defined by~\eqref{err31} and~\eqref{err32} are treated in Lemma~\ref{lem:aux3}. On the other hand, the treatment of the error term $\varepsilon_n^{(3,3)}$ given by~\eqref{err33} is more delicate, since the mapping $g'g$ is not assumed to be Lipschitz continuous on $\R^+$, see the inequality~\eqref{eq:boundg'g}. The analysis of the strong convergence of the numerical scheme is then performed in two steps, in order to obtain the rates of convergence $1/2$ and $1$.

Lemma~\ref{lem:aux3} provides upper bounds for the error terms $\varepsilon_n^{(3,1)}$ and $\varepsilon_n^{(3,2)}$ defined by~\eqref{err31} and~\eqref{err32}. 
\begin{lemma}\label{lem:aux3}
Under the setting of \Cref{theoMain}, one has the following error bounds for the auxiliary error terms $\varepsilon_n^{(3,1)}$ and $\varepsilon_n^{(3,2)}$ defined by~\eqref{err31} and~\eqref{err32}: for all
$n\in\{1,\ldots,N\}$ one has
\begin{equation}\label{eqErr3a}
\E\left[|\varepsilon_n^{(3,j)}|^{2p}\right]\leq C_p(T)\dt\sum_{k=0}^{n-1}\E\left[ |e_k|^{2p}\right].
\end{equation}
\end{lemma}

\begin{proof}[Proof of Lemma~\ref{lem:aux3}]
For the auxiliary error term~\eqref{err31}, since the mapping $G$ is Lipschitz continuous on $\R^+$ owing to Assumption~\ref{assCoeff}, applying the H\"older inequality one obtains the upper bounds
\begin{align*}
\E\left[|\varepsilon_{n}^{(3,1)}|^{2p}\right]&\le C_p(T)\tau\sum_{k=0}^{n-1}\E\left[\big|G(X_k)-G(X(t_k)\big|^{2p}\right]\\
&\le C_p(T)\tau\sum_{k=0}^{n-1}\E\bigl[|e_k|^{2p}\bigr].
\end{align*}
For the auxiliary error term~\eqref{err32}, since the mapping $g$ is Lipschitz continuous on $\R^+$ owing to Assumption~\ref{assCoeff}, applying the discrete-time version~\eqref{eq:BDGdiscrete} of the BDG inequality one obtains the upper bounds
\begin{align*}
\E\left[|\varepsilon_{n}^{(3,2)}|^{2p}\right]&\le C_p(T)\sum_{k=0}^{n-1}\int_{t_k}^{t_{k+1}}\E\left[\big|g(X_k)-g(X(t_k)\big|^{2p}\right]\dd t\\
&\le C_p(T)\tau\sum_{k=0}^{n-1}\E\bigl[|e_k|^{2p}\bigr].
\end{align*}
The proof of Lemma~\ref{lem:aux3} is thus completed.
\end{proof}

We conclude this article by providing the proof of its main result.
\begin{proof}[Proof of Theorem~\ref{theoMain}]

Recalling the decomposition~\eqref{eqDecErr} of the error, and applying the upper bounds from Lemmas~\ref{lem:aux1},~\ref{lem:aux2} and~\ref{lem:aux3}, one obtains the following result: for all $p\in[1,\infty)$, there exists $C_p(T),C_p(T,x_0)\in(0,\infty)$ such that for all $\tau=T/N$ and all $n\in\{0,\ldots,N-1\}$, if $\dt\le \tau_p$, one has
\[
\E\bigl[|e_n|^{2p}\bigr]\le C_p(T)\tau\sum_{k=0}^{n-1}\E\bigl[|e_k|^{2p}\bigr]+C_p(T,x_0)\dt^{2p}+C_p(T)\E[|\varepsilon_n^{(3,3)}|^{2p}].
\]
Applying the discrete Gr\"onwall lemma then yields the inequality
\begin{equation}\label{eq:postGronwall}
\underset{0\le n\le N}\sup~\E\bigl[|e_n|^{2p}\bigr]\le C_p(T)\left(\dt^{2p}+\underset{0\le n\le N}\sup~\E[|\varepsilon_n^{(3,3)}|^{2p}]\right).
\end{equation}
It remains to deal with the error term $\varepsilon_n^{(3,3)}$, defined by~\eqref{err33}.

The first step is to achieve strong order (at least) $1/2$. Note that the random variables $X_k$ and $X(t_k)$ are $\mathcal{F}_{t_k}$-measurable, and that the Brownian increment $B(t)-B(t_k)$ is independent of $\mathcal{F}_{t_k}$. Applying the discrete-time version~\eqref{eq:BDGdiscrete} of the BDG inequality, one obtains
\begin{align}
\E\left[|\varepsilon_{n}^{(3,3)}|^{2p}\right]&\le C_p(T)\sum_{k=0}^{n-1}\int_{t_k}^{t_{k+1}}\E\left[\big|\bigl(\bigl(g'g)(X_k)-\bigl(g'g\bigr)(X(t_k))\bigr)\bigl[B(t)-B(t_k)\bigr]\big|^{2p}\right]\dd t\nonumber\\
&\le C_p(T)\sum_{k=0}^{n-1}\E\left[\big|\bigl(\bigl(g'g)(X_k)-\bigl(g'g\bigr)(X(t_k))\bigr)\big|^{2p}\right]\int_{t_k}^{t_{k+1}}\E[|B(t)-B(t_k)|^{2p}]\dd t\nonumber\\
&\le C_p(T)\tau^{p+1}\sum_{k=0}^{n-1}\E\left[\big|\bigl(\bigl(g'g)(X_k)-\bigl(g'g\bigr)(X(t_k))\bigr)\big|^{2p}\right].\label{eq:goood}
\end{align}
Owing to Assumption~\ref{assCoeff}, the first order derivative $g'$ of the diffusion coefficient is bounded, and the inequality~\eqref{eq:Gg-growth} shows that the mapping $g$ has at most linear growth. As a result, applying the moment bounds~\eqref{momentExact} and~\eqref{momentNum} on the exact and numerical solutions given in \Cref{propMomEx} and \Cref{propMomNum}, one obtains the upper bounds
\begin{align*}
\E\left[\big|\bigl(\bigl(g'g)(X_k)-\bigl(g'g\bigr)(X(t_k))\bigr)\big|^{2p}\right]&\le 2^{2p-1}\E\left[\big|\bigl(\bigl(g'g)(X_k)\big|^{2p}\right]+2^{2p-1}\E\left[\big|\bigl(g'g\bigr)(X(t_k))\bigr)\big|^{2p}\right]\\
&\le 2^{2p}\|g'\|_\infty^{4p}\bigl(\E\left[|X_k|^{2p}\right]+\E\left[|X(t_k)|^{2p}\right]\bigr)\\
&\le C_p(T,x_0).
\end{align*}
As a result, one obtains the bound
\[
\underset{0\le n\le N}\sup~\E\left[|\varepsilon_{n}^{(3,3)}|^{2p}\right]\le C_p(T,x_0)\tau^{p}.
\]
Therefore, substituting the upper bound above in the inequality~\eqref{eq:postGronwall} shows that the strong order of convergence of the numerical scheme~\eqref{eq:scheme} is at least equal to $1/2$: for all $p\in[1,\infty)$, there exists $C_p(T,x_0)\in(0,\infty)$ such that for all $\tau=T/N$ and all $n\in\{0,\ldots,N-1\}$, if $\dt\le \tau_p$, one has
\begin{equation}\label{eq:result1}
\underset{0\le n\le N}\sup~\E\left[|e_n|^{2p}\right]\le C_p(T,x_0)\tau^p.
\end{equation}

The second step of the proof is now performed, where it is proved that strong order $1$ is achieved. The mapping $g'g$ is not assumed to be Lipschitz continuous on $\R^+$, however it satisfies the inequality~\eqref{eq:boundg'g}. Therefore, considering the upper bound~\eqref{eq:goood} on $\E\left[|\varepsilon_{n}^{(3,3)}|^{2p}\right]$ obtained above, and applying the Cauchy--Schwarz inequality, one has
\begin{align*}
\E\left[|\varepsilon_{n}^{(3,3)}|^{2p}\right]
&\le C_p(T)\tau^{p+1}\sum_{k=0}^{n-1}\E\left[\big|\bigl(\bigl(g'g)(X_k)-\bigl(g'g\bigr)(X(t_k))\bigr)\big|^{2p}\right]\\
&\le C_p(T)\tau^{p+1}\sum_{k=0}^{n-1}
\left(\E\left[|e_k|^{4p}\right]\right)^{\frac12}
\left(1+\left(\E\left[|X_k|^{4p}\right]\right)^{\frac12}+\left(\E\left[|X(t_k)|^{4p}\right]\right)^{\frac12}\right).
\end{align*}
Finally, applying the strong error estimates~\eqref{eq:result1} with order $1/2$ in $L^{4p}(\Omega)$ (as $p\in[1,\infty)$ is arbitrary) and the moment bounds~\eqref{momentExact} and~\eqref{momentNum} on the exact and numerical solutions given in \Cref{propMomEx} and \Cref{propMomNum}, one obtains
\[
\underset{0\le n\le N}\sup~\E\left[|\varepsilon_{n}^{(3,3)}|^{2p}\right]\le C_p(T,x_0)\tau^{2p}.
\]
Therefore, substituting the upper bound above in the inequality~\eqref{eq:postGronwall} shows that the strong order of convergence of the numerical scheme~\eqref{eq:scheme} is at least equal to $1$: for all $p\in[1,\infty)$, there exists $C_p(T,x_0)\in(0,\infty)$ such that for all $\tau=T/N$ and all $n\in\{0,\ldots,N-1\}$, if $\dt\le \tau_p$, one has
\begin{equation}\label{eq:result2}
\underset{0\le n\le N}\sup~\E\left[|e_n|^{2p}\right]\le C_p(T,x_0)\tau^{2p}.
\end{equation}
This concludes the proof of \Cref{theoMain}.
\end{proof}

\section*{Acknowledgements}
This work was initiated thanks to the support of the SFVE-A program. 
The work of DC was partially supported by the Swedish Research Council (VR) (projects nr. $2018-04443$
and $2024-04536$) and partially supported by the European Union (ERC, StochMan, 101088589, PI A. Lang).
Views and opinions expressed are however those of the author(s) only and do not necessarily reﬂect those
of the European Union or the European Research Council. Neither the European Union nor the granting authority can be held responsible for them.
The computations were performed on resources provided by
the National Academic Infrastructure for Supercomputing in Sweden (NAISS) at Vera, Chalmers e-Commons
at Chalmers University of Technology and partially funded by the Swedish Research Council
through grant agreement no. 2022-06725.

%\bibliographystyle{plain}
%\bibliography{labib}

\end{document}